\documentclass[12pt,a4paper]{amsart}

\usepackage[top=3cm,left=3cm,right=3cm,bottom=3.5cm]{geometry}
\usepackage{amsmath}
\usepackage{amssymb}
\usepackage{graphicx}
\usepackage{psfrag}

\newtheorem{lem}{Lemma}[section]
\newtheorem{prop}[lem]{Proposition}
\newtheorem{theorem}[lem]{Theorem}

\theoremstyle{definition}
\newtheorem{rem}[lem]{Remark}

\def\C{\mathcal{C}}
\def\D{\mathcal{D}}
\def\F{\mathcal{F}}
\def\H{\mathcal{H}}
\def\K{\mathcal{K}}
\def\PP{\mathcal{P}}
\def\T{\mathcal{T}}
\def\W{\mathcal{W}}
\def\X{\mathcal{X}}

\def\Hom{\operatorname{Hom}}
\def\End{\operatorname{End}}
\def\Ext{\operatorname{Ext}}

\def\add{\operatorname{add}}
\def\ind{\operatorname{ind}}
\def\dim{\operatorname{dim}}
\def\mod{\operatorname{mod}}

\def\ql{\operatorname{ql}}
\def\opp{\operatorname{op}}
\def\op{^{\opp}}

\title{Endomorphism rings of maximal rigid objects in cluster tubes}

\author{Dagfinn F. Vatne}
\address{Institutt for Matematiske Fag\\
Norges Teknisk-Naturvitenskapelige Universitet\\
N-7491 Trondheim\\
Norway}
\email{dvatne@math.ntnu.no}

\begin{document}

\begin{abstract}
We describe the endomorphism rings of maximal rigid objects in the cluster
categories of tubes. Moreover, we show that they are gentle and have
Gorenstein dimension 1. We analyse their representation theory and prove that
they are of finite type. Finally, we study the relationship between the module
category and the cluster tube via the Hom-functor.
\end{abstract}

\maketitle

\section*{Introduction}

Cluster categories were defined in \cite{bmrrt} as tools for categorification
of Fomin-Zelevinsky cluster algebras \cite{fz}. They are defined as the orbit
categories of the derived category $D^b(\H)$ of hereditary abelian categories
$\H$ by a certain autoequivalence.

In the situation where $\H$ is the category of finite dimensional
representations of a finite acyclic quiver, the cluster category has been
subject to intense investigation. In this case it has been shown that the
cluster category and the set of exceptional objects form a good model for the
cluster algebra associated with the same quiver.

In this paper we work with a cluster category $\C_n$ defined from a different
hereditary abelian category, namely the tube $\T_n$. This category is called
the cluster tube and has recently been studied in \cite{bkl1,bkl2} and
\cite{bmv}. Although this category is also a Hom-finite triangulated
2-Calabi-Yau category, it does not admit all of the nice properties of cluster
categories from quivers. In particular, the maximal rigid (also called
maximal exceptional) objects do not satisfy the more restrictive definition of
cluster-tilting objects.

Moreover, the Gabriel quivers of the endomorphism rings of maximal rigid
objects in the cluster tube have loops. Consequently, $\C_n$ with its maximal
rigid objects does not carry a cluster structure in the sense of
\cite{birs}. The axioms for cluster structures can be modified, however, to
apply also to cluster tubes, see \cite{bmv}.

The aim of the present paper is to study the endomorphism rings of the maximal
rigid objects. We will find a description in terms of quivers with relations. 
Like cluster-tilted algebras, the algebras we consider here
are Gorenstein of Gorenstein dimension 1, unless $n=2$, in which case they are
self-injective. However, the proof (from \cite{kr})
for cluster-tilted algebras has no analogy in our setting. Instead, we
use the fact that our algebras are gentle, and apply the technique from
\cite{gr} to our quivers with relations in order to prove the
result. The properties of the algebras we study in this paper are thus
reminiscent of those of the algebras recently studied in \cite{abcp}.

Since the endomorphism rings are gentle, they are string algebras. We use the
theory of string- and band-modules to show that the endomorphism rings are of
finite type. One of the main results about cluster-tilted algebras, which was
proved in \cite{bmr}, is the close connection between the module category of
the cluster-tilted algebra and the cluster category it arises from. This
connection is provided by the Hom-functor. In our situation, the Hom-functor
is not full, and therefore there is no analogous theorem. We will nevertheless
study the action of the Hom-functor on the objects, and in particular show
that it is dense. Indeed, when $T$ is maximal rigid, we find an explicit
description of $\Hom_{\C_n}(T,X)$ for every indecomposable $X$ in $\C_n$.

The paper is organised as follows: Section \ref{sec:objects} contains the
definition of the cluster tube and a description of maximal rigid objects
recalled from \cite{bmv}. In Section \ref{sec:endoring} we give the
description of the endomorphism rings, before we in Section \ref{sec:gentle}
study the gentleness and Gorenstein dimension and give some facts about
indecomposable representations. Finally, in Section \ref{sec:homfunctor} we
describe the action of the Hom-functor.

Throughout the paper we will work over some field $k$, which is assumed to be
algebraically closed. Modules over an algebra will always mean left modules,
and we will read paths in quivers from right to left.

\section{Maximal rigid objects in cluster tubes} \label{sec:objects}

We start off by reviewing some properties of cluster tubes. These
categories have recently been studied in \cite{bkl1,bkl2} and \cite{bmv}, and
more details can be found in these papers.

For any integer $n\geq 2$, let $\T_n$ be a tube of rank $n$, that is, the
category of nilpotent representations of a cyclically oriented
$\tilde{A}_{n-1}$-quiver. It can also be realised as the thick subcategory
generated by a tube in the regular component in the AR-quiver of a suitable
tame hereditary algebra. All maps in this category are linear combinations of
finite compositions of irreducible maps, and are subject to mesh relations in
the AR-quiver.

$\T_n$ is a hereditary abelian category, and following the construction
introduced in \cite{bmrrt}, we form its cluster category, called
the \emph{cluster tube of rank $n$} and denoted $\C_n$. This is by definition
the orbit category obtained from the bounded derived category $\D_n=D^b(\T_n)$
by the action of the self-equivalence $\tau^{-1} \circ [1]$. Here, $[1]$
denotes the suspension functor of $\D_n$, while $\tau$ is the Auslander-Reiten
translation. Unless the actual value of $n$ is important, we will usually
suppress the subscript $n$ in the notation, and write $\T$, $\D$ and $\C$.

For a finite-dimensional hereditary algebra $H$, a theorem due to Keller
\cite{k} guarantees that the associated cluster category $\C_H$ is
triangulated with a canonical triangulated structure inherited from the
derived category. Keller's result is not directly applicable in our situation,
since $\T$ has no tilting objects. Nevertheless, $\C_n$ also inherits a
triangulated structure from $\D_n$, see \cite{bkl1} for a rigorous
treatment of this.

The indecomposable objects of the cluster tube $\C$ are in bijection with
the indecomposables in $\T$ itself, and we will sometimes use the same
symbol to denote both an object in the tube $\T$ and its image in the
cluster tube $\C$. The irreducible maps in $\C$ are the images of the
irreducible maps in $\D$, which again are the shifts of the irreducible maps
in the tube $\T$. So the AR-quiver of $\C$ is isomorphic to the AR-quiver of
$\T$.

For a given rank $n$, we will use a coordinate system on the indecomposable
objects. Choose once and for all a quasisimple object and give it coordinates
$(1,1)$. Now give the other quasisimples coordinates $(q,1)$ such that $\tau
(q,1)=(q-1,1)$, where $q$ is reduced modulo the rank $n$. Then give the
remaining indecomposables coordinates $(a,b)$ in such a way that there are
irreducible morphisms $(a,b)\to (a,b+1)$ for $b\geq 1$ and $(a,b)\to
(a+1,b-1)$ for $b\geq 2$. Throughout, the first coordinates will be reduced
modulo $n$. See Figure \ref{figure:coordinates}.
\begin{figure}
\centering
\includegraphics[width=12cm]{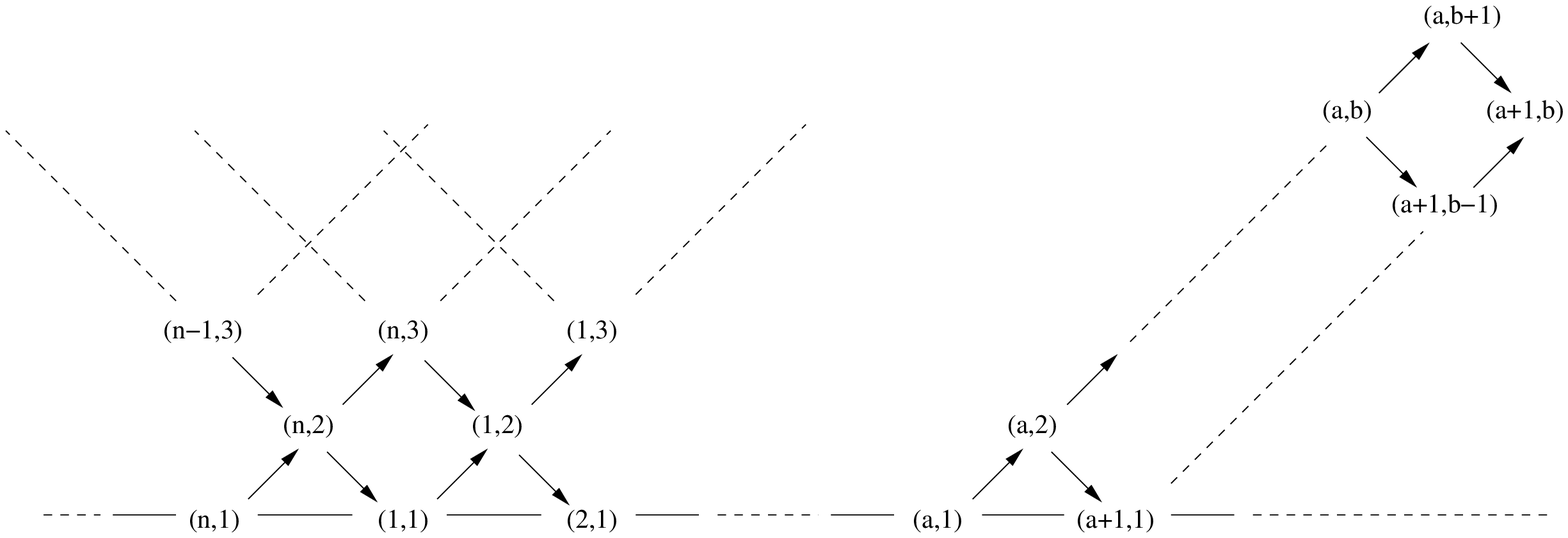}
\caption{AR-quiver and coordinate system for $\T_n$ and
  $\C_n$.} \label{figure:coordinates}
\end{figure}

The infinite sequence of irreducible maps
\[
\mathbf{R}_{(a,i)} = (a,i)\to (a,i+1)\to \cdots \to (a,i+j)\to \cdots
\]
is called a \emph{ray}. Similarly, the infinite sequence
\[
\mathbf{C}_{(a,i)} = \cdots \to (a-b,i+b)\to \cdots \to (a-1,i+1)\to (a,i)
\]
is called a \emph{coray}. Note that the sum of the coordinates is constant,
modulo $n$, for indecomposables located on the same coray.

For an indecomposable object $X=(a,i)$ where $i<n$ we will also need the
notion of the \emph{wing} $\W_X$ determined by $X$. This is by definition the
set of indecomposables whose position in the tube is in the triangle with $X$
on top. $X$ will be called the \emph{summit} of $\W_X$. In terms of
coordinates, the objects in the wing $\W_{(a,i)}$ are $(a',i')$ such that
$a'\geq a$ and $a'+i'\leq a+i$. The \emph{height} of $\W_X$ is the
quasi-length $\ql X$.

The dimensions of the Hom-spaces in $\C$ are given by the following lemma,
proved in \cite{bmv}.

\begin{lem} \label{lem:homspaces}
For $X$ and $Y$ indecomposable in $\C$, we have
\[
\Hom_{\C}(X,Y) \simeq \Hom_{\T}(X,Y) \amalg D\Hom_{\T}(Y,\tau^2 X)
\]
where $D$ is the usual $k$-vector space duality $\Hom_k(-,k)$.
\end{lem}

When $X$ and $Y$ are indecomposable, the maps in $\Hom_{\C}(X,Y)$ which are
images of maps in $\Hom_{\D}(\hat{X},\tau^{-1}\hat{Y}[1])$ for
$\hat{X},\hat{Y}$ in $\T$ will be called $\D$-maps, and those
which are images of maps in $\T$ itself will be called $\T$-maps. Since $\T$
is hereditary, all maps in $\C$ are linear combinations of maps of these two
kinds. The Hom-hammock of an indecomposable object (that is, the support of
$\Hom_{\C}(X,-)$) is illustrated in Figure \ref{figure:homhammock}. Note that
the two components in the figure wrap around the tube and intersect. Moreover,
if $b\geq n+1$, then each component intersects itself, possibly with several
layers, and therefore there exist Hom-spaces of arbitrary finite dimension
between indecomposables.
\begin{figure}
\centering
\includegraphics[width=9cm]{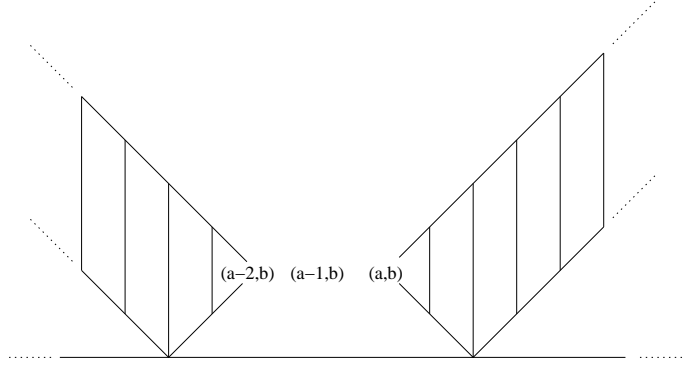}
\caption{The Hom-hammock of $(a,b)$. There are $\T$-maps to the
  indecomposables in the right component, and $\D$-maps to the indecomposables
in the left component.} \label{figure:homhammock}
\end{figure}

So for indecomposable $X$ and $Y$, the existence of a $\D$-map $X\to Y$ is
equivalent to the existence of a $\T$-map $Y\to \tau^2 X$. The following lemma
is then easily verified:

\begin{lem} \label{lem:Dendo}
Let $X$ be an indecomposable object of $\C_n$. Then there exists a
$\D$-endomorphism of $X$ if and only if $\ql X\geq n-1$.
\end{lem}

We will need the following lemma on the relationship between $\T$-maps and
$\D$-maps:

\begin{lem} \label{lem:rayfactoring}
For $X$, $Y$ and $Z$ indecomposable objects in $\C_n$, we have the following:
\begin{itemize}
\item[(i)] Assume that $\ql X\leq n$ and $\ql Y\leq n$. If there are non-zero
  $\D$-maps $\psi_{XZ}:X\to Z$ and $\psi_{YZ}:Y\to Z$, and an irreducible map
  $i_{XY}:X\to Y$, then $\psi_{YZ}\circ i_{XY}=\psi_{XZ}$ up to multiplication
  by a non-zero scalar.
\item[(ii)] Assume that $\ql X\leq n$. If there are non-zero $\D$-maps
  $\psi_{XY}:X\to Y$ and $\psi_{XZ}:X\to Z$, and an irreducible map
  $i_{YZ}:Y\to Z$, then $\psi_{XZ}=i_{YZ}\circ \psi_{XY}$, up to
  multiplication by a non-zero scalar.
\end{itemize}
\end{lem}

\begin{rem} \label{rem:rayfactoring}
Note that by repeated application, the same applies to compositions of
irreducible maps, i.\ e.\ to all $\T$-maps, under the assumption that the
required Hom-spaces are non-zero for each indecomposable that the composition
factors through.
\end{rem}

\begin{proof}[Proof of Lemma \ref{lem:rayfactoring}]
\textbf{(i):} We lift the maps to the derived category $\D$, and denote by
$\hat{X}$, $\hat{Y}$ and $\hat{Z}$ the preimages of the objects in $\T$. Since
$X$ and $Y$ have quasilength $\leq n$, we have that the space
$\Hom_{\D}(\hat{X},\tau^{-1}\hat{A}[1])$ of $\D$-maps is at most
one-dimensional for any indecomposable $\hat{A}\in \T$, and similarly for $Y$.

The aim is to show that the map
\[
\Hom_{\D}(i_{XY},\tau^{-1}\hat{Z}[1]):
\Hom_{\D}(\hat{Y},\tau^{-1}\hat{Z}[1]) \to
\Hom_{\D}(\hat{X},\tau^{-1}\hat{Z}[1])
\]
is surjective. We can view this as a map
\[
i_{XY}^*:\Ext^1_{\T}(\hat{Y},\tau^{-1}\hat{Z}) \to
\Ext^1_{\T}(\hat{X},\tau^{-1}\hat{Z})
\]
or, by duality and the AR-formula,
\[
\Hom_{\T}(\tau^{-1}\hat{Z},\tau i_{XY}): \Hom_{\T}(\tau^{-1}\hat{Z},\tau
\hat{X}) \to \Hom_{\T}(\tau^{-1}\hat{Z},\tau \hat{Y})
\]
which we now wish to show is injective. But this is clear from the structure
of the tube when the Hom-spaces are non-zero.

\textbf{(ii):} We need to show that the map
\[
\Hom_{\D}(\hat{X},\tau^{-1}i_{YZ}[1]): \Hom_{\D}(\hat{X},\tau^{-1}\hat{Y}[1])\to
\Hom_{\D}(\hat{X},\tau^{-1}\hat{Z}[1])
\]
is surjective. Similarly as above, by duality this is equivalent to the map
\[
\Hom_{\T}(\tau^{-1}i_{YZ},\tau \hat{X}): \Hom_{\T}(\tau^{-1}\hat{Z},\tau
\hat{X})\to \Hom_{\T}(\tau^{-1}\hat{Y},\tau \hat{X})
\]
being injective. But by the combinatorics of the tube, this is clearly an
isomorphism, since by assumption both spaces are 1-dimensional.
\end{proof}

Let $1\leq h\leq n-1$, and choose some indecomposable $X$ in $\T_n$ with
quasilength $\ql X=h$. Let $\vec{A}_h$ be a linearly oriented quiver with
underlying graph the Dynkin diagram $A_h$. Then the category $\mod k\vec{A}_h$
of finitely generated modules over the path algebra $k\vec{A}_h$ is naturally
equivalent with the subcategory $\add_{\T} \W_X$ of $\T_n$. Embedding to
$\D_n$ and projecting to $\C_n$ we find that $\mod k\vec{A}_h$ embeds into the
subcategory $\add_{\C} \W_X$ of $\C_n$. The image is the subcategory
$\add^{\T}_{\C} \W_X$ obtained by deleting the $\D$-maps from $\add_{\C}
\W_X$. From now on, we will drop the subscript when we speak of an additive
hull as a set of objects, since there is a bijection between the objects of
$\T$ and those of $\C$.

The triangulated category $\C$ is a 2-Calabi-Yau category, which in particular
means that for any two objects $X$ and $Y$, we have symmetric Ext-spaces:
\[
\Ext^1_{\C}(X,Y)\simeq D\Ext^1_{\C}(Y,X)
\]
Two indecomposable objects $X$ and $Y$ will be called \emph{compatible} if
$\Ext^1_{\C}(X,Y)=\Ext^1_{\C}(Y,X)=0$. It is worth noticing that
$X$ and $Y$ are compatible if and only if
$\Ext^1_{\T}(X,Y)=\Ext^1_{\T}(Y,X)=0$.

In an abelian or triangulated category $\K$, an object $T$ is called
\emph{rigid} if it satisfies $\Ext^1_{\K}(T,T)=0$. If it is maximal with respect
to this property, that is if $\Ext^1_{\K}(T\amalg X,T\amalg X)=0$ implies that
$X\in \add T$, then it is called \emph{maximal rigid}. The maximal rigid
objects in the cluster tube $\C$ do not satisfy the stronger condition of
\emph{cluster tilting}, see \cite{bmv}.

The following description of the maximal rigid objects was given in
\cite{bmv}:

\begin{prop} \label{prop:tiltedbijection}
There is a natural bijection between the set of maximal rigid objects in
$\C_n$ and the set
\[
\left\{ \textrm{tilting modules of }k\vec{A}_{n-1} \right\} \times
\left\{ 1,...,n\right\}
\]
where $\vec{A}_{n-1}$ is a linearly oriented quiver with the Dynkin diagram
$A_{n-1}$ as its underlying graph.
\end{prop}

The proposition follows from the following considerations, which will be
needed for the rest of the paper: All summands of a
maximal rigid object in $\C_n$ are concentrated in the wing $\W_{T_1}$
determined by a \emph{top summand} $T_1$ with $\ql T_1 =n-1$. Now the claim
follows from the embedding of $\mod k\vec{A}_{n-1}$ into $\add_{\C}\W_{T_1}$,
since it is easily seen that $\Ext_{\C_n}^1(X,Y)$ for two
indecomposables $X$ and $Y$ in $\W_{T_1}$ vanishes if and only if both
$\Ext_{k\vec{A}_{n-1}}^1(\widetilde{X},\widetilde{Y})$ and
$\Ext_{k\vec{A}_{n-1}}^1(\widetilde{Y},\widetilde{X})$ vanish, where
$\widetilde{X}$ and $\widetilde{Y}$ are the corresponding
$k\vec{A}_{n-1}$-modules. Since there are $n$ choices for the top summand, this
provides the bijection.

\section{The endomorphism rings} \label{sec:endoring}

With the description of the maximal rigid objects of $\C$ presented in
Section \ref{sec:objects}, we now proceed to determine their endomorphism
rings in terms of quivers and relations.

Let $T$ be a maximal rigid object in the cluster tube $\C_n$, and let $M_T$
denote the tilting $k\vec{A}_{n-1}$-module associated with $T$ according to
Proposition \ref{prop:tiltedbijection}. Since the module category of a
hereditary algebra $H$ sits naturally embedded in the cluster category $\C_H$,
we can think of the module $M_T$ as a cluster-tilting object in
$\C_{A_{n-1}}$. The endomorphism ring, or cluster-tilted algebra,
$\widetilde{\Gamma}_T = \End_{\C_{A_{n-1}}}(M_T)\op$ can easily be found from
the tilted algebra $\Gamma_T = \End_{k\vec{A}_{n-1}}(M_T)\op$, by the results in
\cite{bre}, or more generally in \cite{abs}.

Every minimal relation on the quiver of a tilted algebra of type $A$ is a zero
relation of length two. The quiver of the cluster-tilted algebra is then
obtained by inserting an arrow $\alpha_{\rho}$ from the end vertex to the
start vertex of each defining relation path $\rho$. The relations for the
cluster-tilted algebra are, as prescribed by \cite{bmr2}, the compositions of
any two arrows in any of the 3-cycles formed by adding the new arrows.

We can now formulate the main theorem of this section.

\begin{theorem} \label{theorem:endoring}
Let $T$ be a maximal rigid object in $\C_n$. Then the endomorphism ring
$\Lambda_T = \End_{\C_n}(T)\op$ is isomorphic to the algebra $kQ/I$ where
\begin{itemize}
\item[(a)] $Q$ is the quiver obtained from the quiver of
  $\widetilde{\Gamma}_T$ by adjoining a loop $\omega$ to the vertex
  corresponding to the projective-injective $k\vec{A}_{n-1}$-module.
\item[(b)] $I$ is the ideal generated by the relations in
  $\widetilde{\Gamma}_T$ and in addition $\omega^2$.
\end{itemize}
\end{theorem}

Before we can present the proof of the theorem, we need some considerations on
the combinatorial structure of the maximal rigid objects, so we postpone the
proof until we have established a few lemmas.

We define a \emph{non-degenerate subwing triple} $(X;Y,Z)$ to be a triple
$X,Y,Z$ of indecomposables in $\C$ with $3\leq \ql X \leq n-1$ such that if
$X=(a,b)$, then $Y=(a,c)$ and $Z=(a+c+1,b-c-1)$ for some $1\leq c\leq
b-2$. This means that $X$ is on the ray $\mathbf{R}_Y$ and on the coray
$\mathbf{C}_Z$, so in particular $\W_Y$ and $\W_Z$ are contained in
$\W_X$. Moreover, there is exactly one quasisimple $(a+c,1)$ which is in
$\W_X$ but not in $\W_Y\cup \W_Z$. See Figure \ref{figure:subwingtriple}. A
\emph{degenerate subwing triple} $(X;Y,Z)$ is a triple with $2\leq \ql X\leq
n-1$ such that if $X=(a,b)$, then either $Y=(a,b-1)$ and $Z=0$ or $Y=0$ and
$Z=(a+1,b-1)$. Note that any subwing triple (degenerate or non-degenerate) is
determined by the top indecomposable $X$ and the unique quasisimple which is
not in any of the two subwings $\W_Y$ or $\W_Z$.
\begin{figure}
\centering
\includegraphics[width=10cm]{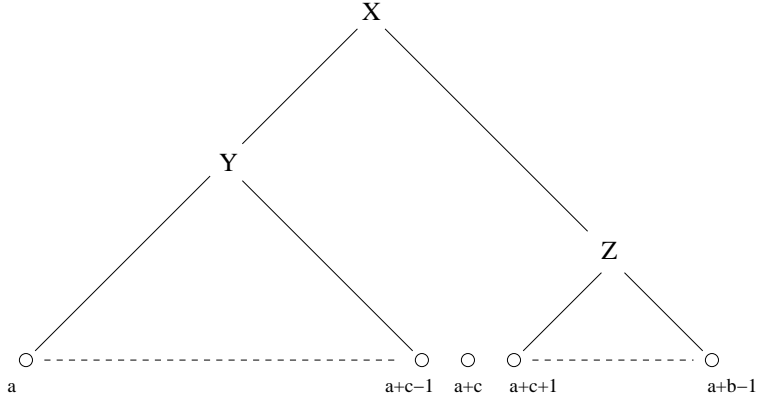}
\caption{Non-degenerate subwing triple $(X;Y,Z)$. If $X=(a,b)$ and
  $Y=(a,c)$ with $1\leq c\leq b-2$, then
  $Z=(a+c+1,b-c-1)$.} \label{figure:subwingtriple}
\end{figure}

\begin{lem} \label{lem:subwingmaps}
Let $(X;Y,Z)$ be a non-degenerate subwing triple. Let $Y'\in \W_Y$ and $Z'\in
\W_Z$.
\begin{itemize}
\item[(i)] There are no $\T$-maps $Z'\to Y'$.
\item[(ii)] There are no $\T$-maps $Y'\to Z'$.
\item[(iii)] There is a $\D$-map $Z'\to Y'$ if and only if $Z'$ is on the left
  edge of $\W_Z$ and $Y'$ is on the right edge of $\W_Z$. In this case, this
  map factors through the $\D$-map $Z\to Y$.
\item[(iv)] There is a $\D$-map $Y'\to Z'$ if and only if $Z'$ is on the right
  edge of $\W_Z$ and $Y'$ is on the left edge of $\W_Y$, and $\ql X=n-1$. In
  this case, this map factors through the $\D$-endomorphism of $X$.
\end{itemize}
\end{lem}

\begin{proof}
Claims (i) and (ii) are easily verified; one must keep in mind that $\ql X\leq
n-1$ by the definition of subwing triples.

Since the existence of a $\D$-map $Z'\to Y'$ is equivalent to the existence of
a $\T$-map $Y'\to \tau^2 Z'$, we see that the only way such a map can arise is
when $Z'$ is on the left edge of $\W_Z$, and $Y'$ is on the right edge of
$\W_Y$. Now by Lemma \ref{lem:rayfactoring} and Remark \ref{rem:rayfactoring},
this $\D$-map factors through the ray $\mathbf{R}_{Z'}$. In particular, it
factors $Z'\overset{\phi_{Z'Z}}{\to} Z\overset{\psi_{ZY'}}{\to} Y'$ where
$\phi_{Z'Z}$ is the $\T$-map from $Z'$ to $Z$, and $\psi_{ZY'}$ is the unique
(up to multiplication with scalars) $\D$-map $Z\to Y'$. Applying Lemma
\ref{lem:rayfactoring} to $\psi_{ZY'}$, we find that it factors
$Z\overset{\psi_{ZY}}{\to} Y\overset{\phi_{YY'}}{\to} Y'$, where $\psi_{ZY}$
is the $\D$-map from $Z$ to $Y$ and $\phi_{YY'}$ is the $\T$-map from $Y$ to
$Y'$. So claim (iii) holds.

For claim (iv), observe that since $\ql X\leq n-1$, a necessary
condition for the existence of a $\T$-map $Z'\to \tau^2 Y'$ is that $\ql
X=n-1$. Moreover, we have that $Z'$ must be on the right edge of $\W_X$ and
$Y'$ must be on the left edge of $\W_X$. Now the claim is proved using a
similar argument as for (iii) and the fact from Lemma \ref{lem:Dendo} that if
$\ql X=n-1$ then $X$ has a $\D$-endomorphism.
\end{proof}

\begin{lem} \label{lem:subwingext}
Let $(X;Y,Z)$ be a subwing triple, and let $W\in \W_X$.
\begin{itemize}
\item[(i)] $Y$ and $W$ are compatible if and only if $W\in \W_Y\cup \W_Z$ or
  $W\in \mathbf{R}_Y$.
\item[(ii)] $Z$ and $W$ are compatible if and only if $W\in \W_Y\cup \W_Z$ or
  $W\in \mathbf{C}_Z$.
\end{itemize}
\end{lem}

\begin{proof}
By the 2-Calabi-Yau property, we have symmetric Ext-groups, so it is enough to
check vanishing of $\Ext^1_{\C}(W,Y)$ and $\Ext^1_{\C}(W,Z)$. For this, we use
the AR-formula
\[
\Ext^1_{\C}(W,Y)\simeq D\Hom_{\C}(Y,\tau W)
\]
and similarly for $Z$. Then consider the intersection of the Hom-hammock of
$Y$ with $\tau \W_X$.
\end{proof}

\begin{lem} \label{lem:maxcompatible}
Let $\W=\W_X$ be a wing in $\C_n$ of height $h<n$, and let $\X$ be a set of
pairwise compatible indecomposable objects in $\W$.
\begin{itemize}
\item[(i)] $\X$ has at most $h$ elements.
\item[(ii)] If $\X$ has less than $h$ elements, there exists a set
  $\widetilde{\X}$ of $h$ pairwise compatible indecomposable objects,
  containing $\X$.
\item[(iii)] If $\X$ has $h$ elements, then $X$ is an element of $\X$.
\end{itemize}
\end{lem}

\begin{proof}
The argument is essentially the same as for Proposition
\ref{prop:tiltedbijection}, and the result follows from the theory of tilting
modules applied to $k\vec{A}_h$-modules, noting that the projective-injective
indecomposable is a summand of every tilting module.
\end{proof}

\begin{lem} \label{lem:subwingstructure}
Let $T_k$ be some indecomposable summand of a maximal rigid object
$T=\amalg_{i=1}^{n-1}T_i$ in $\C_n$. 
\begin{itemize}
\item[(i)] There are $\ql T_k$ indecomposable summands of $T$ in $\W_{T_k}$.
\item[(ii)] Assume $\ql T_k>1$. Then there is a subwing triple
  $(T_k;T_{k'},T_{k''})$ such that $T_{k'}$ and $T_{k''}$ are either summands
  of $T$ or zero, and all summands of $T/T_k$ which are in $\W_{T_k}$ are in
  $\W_{T_{k'}}\cup \W_{T_{k''}}$, with $\ql T_{k'}$ summands in $\W_{T_{k'}}$
  and $\ql T_{k''}$ summands in $\W_{T_{k''}}$.
\end{itemize}
\end{lem}

\begin{proof}
Given a rank $n$, we let $T=\amalg_{i=1}^{n-1}$ be a maximal rigid object in
$\C_n$. We proceed by reverse induction on the quasilength of the summands of
$T$. (That is, from summands of larger quasilength to summands of smaller
quasilength.) This is OK, since there is a (unique) summand of maximal
quasilength.

In accordance with Proposition \ref{prop:tiltedbijection}, consider the top
summand $T_1$ of $T$, which has quasilength $n-1$. Without loss of
generality assume that this has coordinates $(1,n-1)$. The claim (i) holds for
$T_1$ by the proof of Proposition \ref{prop:tiltedbijection}.

Assume first that there are no summands of $T$ among the objects
$(1,i)$, where $i=1,...,n-2$. Then all $n-2$ summands of $T/T_1$ are in
$\W_{(2,n-2)}$. So by Lemma \ref{lem:maxcompatible}, the object $(2,n-2)$ must
be a summand of $T$. Hence in this situation claim (ii) also holds, with the
triple $(T_1;0,(2,n-2))$. A similar argument shows that none of the objects
$(i,n-i)$, where $2\leq i\leq n-1$, are summands if and only if $(1,n-2)$ is a
summand, and thus (ii) holds with $(T_1;(1,n-2),0)$.

Suppose therefore that $(1,n-2)$ is not a summand, and that there is at least
one summand of $T$ with coordinates $(1,i)$, where $1\leq i\leq n-3$. Let
$T_2=(1,i_0)$ be the one of these with highest quasi-length (that is, maximal
$i$). Consider the subwing triple $(T_1;T_2,X)$ where $X=(i_0+2,n-i_0-2)$. By
Lemma \ref{lem:subwingext} and the maximality of $i_0$, all $n-2$ summands of
$T/T_1$ must be in $\W_{T_1}\cup \W_X$. Then it follows from Lemma
\ref{lem:maxcompatible} that there must be $i_0$ summands in $\W_{T_1}$ and
$n-i_0-2$ summands in $\W_X$ and moreover that $X$ is indeed a summand of
$T$. We conclude that both claims (i) and (ii) hold for the top summand.

Assume now that $T_k$ is some summand of $T$ with $\ql T_k>1$, but not the top
summand. Let $T_{j}$ be a summand of $T$ of smallest quasilength with $\ql
T_j>\ql T_k$ such that $T_k\in \W_{T_j}$. (Such a summand exists, since all
summands are in $\W_{T_1}$.) Then by induction, the claims hold for $T_j$, and
by the minimality in the choice of $T_j$, the subwing triple corresponding to
$T_j$ is $(T_j;T_{j'},T_{j''})$ where either $T_{j'}$ or $T_{j''}$ is $T_k$,
and the other one is also a summand. In any case, since by the induction
hypothesis claim (ii) holds for $T_j$, there are $\ql T_k$ summands of $T$ in
$\W_{T_k}$, and so (i) holds for $T_k$.

Now that we know that there are $\ql T_k-1$ summands of $T/T_k$ in $\W_{T_k}$,
we can prove that (ii) holds by the same arguments as for $T_1$ above.
\end{proof}

A subwing triple $(T_k;T_{k'},T_{k''})$ such that $T_k,T_{k'}$ and $T_{k''}$
are summands of a maximal rigid object $T$ will be called a $T$-\emph{subwing
  triple}.

With Lemma \ref{lem:subwingstructure} we have obtained a combinatorial
description of the maximal rigid objects as a system of subwing triples
partially ordered by inclusion. Note that for a maximal rigid $T$ with top
summand $T_1$ there is a natural map from the objects in $\W_{T_1}$ to the
summands of $T$ given by sending an object $X$ to the summand $T_x$ with
smallest quasilength such that $X\in \W_{T_x}$. The restriction of this map to
the set of quasisimples is a bijection. See Figure \ref{fig:example} a) for an
example. Also, we have the following:

\begin{lem} \label{lem:subwingintersection}
Let $T_i,T_j$ be summands of a maximal rigid $T$. Then either $T_i \in
\W_{T_j}$, or $T_j\in \W_{T_i}$, or $\W_{T_i}$ and $\W_{T_j}$ have
empty intersection.
\end{lem}

\begin{proof}
Suppose $T_i\not \in \W_{T_j}$ and $T_j \not \in \W_{T_i}$. Let $T_k$ be a
summand of $T$  such that both $T_i\in \W_{T_k}$ and $T_j\in \W_{T_k}$, and
which has minimal quasilength among summands with this property. There is
some non-degenerate $T$-subwing triple $(T_k;T_{k'},T_{k''})$, and by Lemma
\ref{lem:subwingstructure}, the summands $T_i$ and $T_j$ must be in
$\W_{T_{k'}}\cup \W_{T_{k''}}$. Now by the minimality in the choice of $T_k$,
we know that $T_i$ must be in $\W_{T_{k'}}$ and $T_j$ in $\W_{T_{k''}}$ or
vice versa. It follows that $\W_{T_i}$ and $\W_{T_j}$ have zero intersection,
since $\W_{T_{k'}}$ and $\W_{T_{k''}}$ have empty intersection.
\end{proof}

\begin{lem} \label{lem:Tmapsonrays}
Let $T_i$ and $T_j$ be summands of a maximal rigid $T$. Then the following are
equivalent.
\begin{itemize}
\item[a)] There exists a non-zero $\T$-map $T_i \to T_j$.
\item[b)] Either $T_j$ is on the ray $\mathbf{R}_{T_i}$ or $T_i$ is on the
  coray $\mathbf{C}_{T_j}$.
\end{itemize}
\end{lem}

\begin{proof}
This follows from the observation that since $T_i$ and $T_j$ are
Ext-orthogonal, there is no map $T_i \to \tau T_j$.
\end{proof}

\begin{lem} \label{lem:TsubwingTmaps}
Let $T$ be a maximal rigid object.
\begin{itemize}
\item[(i)] If $(T_i;T_j,T_k)$ is a non-degenerate $T$-subwing triple, there is a
  $\T$-map $f_{ji}:T_j \to T_i$ and a $\T$-map $f_{ik}:T_i \to T_k$, and these
  maps are irreducible in $\add_{\C} T$.
\item[(ii)] If $(T_i;T_j,0)$ a degenerate $T$-subwing triple, there is a
  $\T$-map $f_{ji}:T_j \to T_i$, and this is irreducible in $\add_{\C}
  T$. A similar statement holds for a degenerate $T$-subwing triple
  $(T_i;0,T_k)$.
\item[(iii)] There are no other irreducible $\T$-maps in $\add_{\C} T$ than
  those described in (i) and (ii).
\item[(iv)] If $(T_i;T_j,T_k)$ is non-degenerate, the composition
  $f_{ik}\circ f_{ji}$ is zero.
\end{itemize}
\end{lem}

\begin{proof}
Claims (i) and (ii) are obvious, so consider summands $T_x,T_y$ and assume
that there is a $\T$-map $f_{xy}:T_x\to T_y$. By Lemma \ref{lem:Tmapsonrays},
either the summand $T_y$ must be on the ray $\mathbf{R}_{T_x}$ or $T_x$ must
be on the coray $\mathbf{C}_{T_y}$. Assume the former case, so $T_x$ is on the
left edge of $\W_{T_y}$. We know from Lemma \ref{lem:subwingstructure} that
there is some $T$-subwing triple $(T_z;T_x,T_{x'})$, where $T_x$ must
necessarily be on the left edge of $\W_{T_z}$. (If it was on the right, it
would violate Lemma \ref{lem:subwingintersection}.) Then clearly $f_{xy}$
factors through the $\T$-map $f_{xz}:T_x \to T_z$. A similar argument can be
given in the case where $T_x$ is on the coray $\mathbf{C}_{T_y}$ by
considering a $T$-subwing triple $(T_x;T_{x'},T_{x''})$, so (iii) holds as
well.

The claim (iv) also follows from Lemma \ref{lem:Tmapsonrays}.
\end{proof}

\begin{lem} \label{lem:TsubwingDmaps}
Let $T_i,T_j$ be summands of a maximal rigid $T$. Then the following are
equivalent.
\begin{itemize}
\item[a)] There is a non-zero $\D$-map $T_i\to T_j$.
\item[b)] One of the following is satisfied
\begin{itemize}
\item[b')] There is some non-degenerate $T$-subwing triple $(T_x;T_y,T_z)$
  such that $T_i$ is on the left edge of $\W_{T_z}$ and $T_j$ is on the right
  edge of $\W_{T_y}$.
\item[b'')] $T_i$ is on the left edge of $\W_{T_1}$ and $T_j$ is on the right
  edge of $\W_{T_1}$, where $T_1$ is the top summand.
\end{itemize}
\end{itemize}
\end{lem}

\begin{proof}
Assume first that $T_i=T_1$, the top summand. Then there is a $\D$-map $T_i\to
T_j$ if and only if $T_j$ is on the right edge of $\W_{T_1}$, which is a
special case of b''). Similarly, if $T_j=T_1$, then there is a $\D$-map
$T_i\to T_j$ if and only if $T_i$ is on the left edge of $\W_{T_1}$. So under
the assumption that at least one of $T_i,T_j$ is the top summand, the claim
holds.

Consider therefore the case where neither $T_i$ nor $T_j$ is the top
summand. It is easily seen that if $T_j\in \W_{T_i}$ or $T_i\in \W_{T_j}$,
there can be no $\D$-map $T_i\to T_j$. So consider the summand $T_x$ of
minimal quasilength such that both $T_i$ and $T_j$ are in $\W_{T_x}$. By Lemma
\ref{lem:subwingstructure} there is a $T$-subwing triple
$(T_x;T_y,T_z)$, and by the minimality of $T_x$, either $T_i\in \W_{T_y}$ and
$T_j\in \W_{T_z}$ or the other way around. The claim now follows from Lemma
\ref{lem:subwingmaps}.
\end{proof}

We can now prove the main theorem.

\begin{proof}[Proof of Theorem \ref{theorem:endoring}]
In the argument, we will implicitly use the fact that for
indecomposable objects $X$ and $Y$ with $\ql X, \ql Y\leq n-1$, there is up to
multiplication by scalars at most one $\T$-map and one $\D$-map from $X$ to
$Y$.

Our task is to determine the quiver $Q$ and defining relations of $\Lambda_T
= \End_{\C}(T)\op$. First recall that the functor $\Hom_{\C}(T,-):\C \to
\mod \Lambda_T$ induces an equivalence between $\add_{\C} T$ and the category
$\PP (\Lambda_T)$ of projective $\Lambda_T$-modules. The vertices in $Q$ are
therefore in bijection with the indecomposable summands of $T$, and the arrows
correspond to maps which are irreducible in $\add_{\C} T$.

No $\T$-map can factor through a $\D$-map, so let us first consider the arrows
in $Q$ coming from $\T$-maps and their relations, that is the endomorphism
ring of $T$ as an object in the subcategory $\add^{\T}_{\C}\W_{T_1}$, where
$T_1$ is the top summand of $T$. By virtue of the equivalence of this category
with the module category $\mod k\vec{A}_{n-1}$, we know that this will be the
quiver with relations for the tilted algebra $\Gamma_T=\End_{A_{n-1}}(M_T)\op$.

By Lemma \ref{lem:TsubwingTmaps} part (i), for each non-degenerate $T$-subwing
triple $(T_i;T_j,T_k)$ there exist $\T$-maps $f_{ji}:T_j\to T_i$ and
$f_{ik}:T_i\to T_k$ which are irreducible in $\add_{\C} T$ and hence correspond 
to arrows $\alpha_{ij}:i\to j$ and $\alpha_{ki}:k\to i$ in $Q$. Similarly, by
part (ii), for each degenerate $T$-subwing triple $(T_i;T_j,0)$ there is an
arrow $\alpha_{ij}:i\to j$, and for a $T$-subwing triple $(T_i;0,T_k)$ there
is an arrow $\alpha_{ki}:k\to i$. Moreover, by part (iii) of the
same lemma, these are the only arrows in $Q$ coming from $\T$-maps. Assuming
that the triple is non-degenerate, by part (iv), the composition $f_{ik}\circ
f_{ji}$ is zero, and hence $\alpha_{ij}\alpha_{ki}$ is a zero relation for the
quiver.

It follows from Lemma \ref{lem:Tmapsonrays} that there are no other minimal
relations on the arrows coming from $\T$-maps, since a path
$\alpha_{i_ki_{k-1}}\cdots \alpha_{i_1i_2}\alpha_{i_0i_1}$ such that no
$\alpha_{i_li_{l+1}}\alpha_{i_{l-1}i_l}$ comes from a composition
$f_{i_li_{l-1}}\circ f_{i_{l+1}i_l}$ from a non-degenerate $T$-subwing triple
$(T_{i_l};T_{i_{l+1}},T_{i_{l-1}})$ corresponds to a map following a ray or a
coray. So the above gives a description of the tilted algebra
$\Gamma_T=\End_{A_{n-1}}(M_T)\op$. See Figure \ref{fig:example} b) for an
example.

Now consider the $\D$-maps, and postpone for a moment the situation with maps
to or from the top summand $T_1$. By Lemma \ref{lem:subwingmaps} part (iii),
for each non-degenerate $T$-subwing triple $(T_i;T_j,T_k)$ there is a $\D$-map
$g_{kj}:T_k\to T_j$. We claim that this map is irreducible in $\add_{\C} T$. So
assume that there exists a summand $T_x$, not isomorphic to $T_k$ or $T_j$,
and a $\D$-map $g_{xj}:T_x\to T_j$ such that $f_{kj}=g_{xj}\circ h_{kx}$ where
$h_{kx}:T_k\to T_x$. Since the composition of two $\D$-maps is zero, $h_{kx}$
must be a $\T$-map. So by Lemma \ref{lem:Tmapsonrays} we either have that
$T_x$ is on $\mathbf{R}_{T_k}$, or that $T_k$ is on $\mathbf{C}_{T_x}$. The
former case contradicts Lemma \ref{lem:subwingintersection}, since $T_k$ is on
the right side of $\W_{T_i}$ and on the left side of $\W_{T_x}$. The latter
case contradicts Lemma \ref{lem:subwingmaps}, since $T_x$ is on the right edge
of $\W_{T_k}$. A similar argument shows that we cannot have a factorisation
$g_{kj}=h_{yj}\circ g_{ky}$ where $h_{yj}$ is a $\T$-map and $g_{ky}$ is a
$\D$-map.

Thus $g_{kj}$ is irreducible and corresponds to an arrow $\beta_{jk}:j\to k$
in $Q$.

For the non-degenerate $T$-subwing triple $(T_i;T_j,T_k)$, the composition
$g_{kj}\circ f_{ik} : T_i \to T_j$ is a $\D$-map. We claim that it must be
zero. Assume therefore that there is a $\T$-map from $T_j$ to $\tau^2
T_i$. Then, since $T_i$ sits on $\mathbf{R}_{T_j}$ and there is no $\T$-map
$T_j\to \tau T_i$ by the Ext-orthogonality of $T_i$ and $T_j$, there must be
$\T$-maps from $T_j$ to \emph{all} indecomposables of quasilength $\ql T_i$,
\emph{except} $\tau T_i$. But this would require $\ql T_j = n-1$, and this is
impossible, since $\ql T_j< \ql T_1=n-1$. We conclude that there is no
$\D$-map $T_i \to T_j$, and the composition is zero.

Similarly, the composition $f_{ji}\circ g_{kj} : T_k \to T_i$ is also zero,
since there is no $\T$-map from $T_i$ to $\tau^2 T_k$. It follows that the
paths $\alpha_{ki}\beta_{jk}$ and $\beta_{jk}\alpha_{ij}$ are zero relations
on the quiver $Q$.

By Lemma \ref{lem:Dendo}, there is a $\D$-map $h_{T_1}$ which is an
endomorphism of the top summand $T_1$. This map must be irreducible in
$\add_{\C} T$, for the only objects in $\W_{T_1}$ to which there are maps from
$T_1$ are the ones on the right edge of $\W_{T_1}$, but there are no maps from
any of these to $T_1$. Thus there is a loop $\omega$ on the vertex
corresponding to the top summand.

The composition $h_{T_1}\circ h_{T_1}$ is zero, since any composition of two
$\D$-maps is the image of a map $\T \to \T[2]$ in $\D$, which must necessarily
be zero since $\T$ is hereditary. So $\omega^2$ is a zero relation on the
quiver of $\Lambda_T$.

By Lemmas \ref{lem:subwingmaps} and \ref{lem:TsubwingDmaps}, there are no
other irreducible $\D$-maps between summands of $T$, and also no other
minimal relations involving arrows from $\D$-maps: Let $\beta_{jk}$ be an
arrow corresponding to a $\D$-map $T_k\to T_j$ where $(T_i;T_j,T_k)$ is a
non-degenerate subwing triangle. Then if $\gamma_1\cdots \gamma_k
\beta \gamma^*_1\cdots \gamma^*_{k'}$ is a path in $Q$ such that
$\gamma_1\cdots\gamma_k$ and $\gamma^*_1\cdots \gamma^*_{k'}$ do not traverse
any relations, then necessarily $\gamma_1,...,\gamma_k$ are arrows coming from
maps on the left edge of $\W_{T_k}$ and $\gamma^*_1,...,\gamma^*_{k'}$ are
arrows coming from the right edge of $\W_{T_j}$. Then this path
corresponds to a non-zero map by Lemma \ref{lem:TsubwingDmaps}.

We now see that the arrows $\beta_{xy}$ are in
bijection with the zero relations for the quiver of $\Gamma_T$, and complete
the relation paths to oriented cycles, so by \cite{abs} or \cite{bre} the
arrows $\alpha_{zw}$ and $\beta_{xy}$ form the quiver of
$\widetilde{\Gamma}_T=\End_{\C_{A_{n-1}}}(M_T)\op$. Furthermore, the relations
imposed on this quiver coincide with the relations defining
$\widetilde{\Gamma}_T$.

So the quiver of $\Lambda_T$ is obtained from the quiver of
$\widetilde{\Gamma}_T$ by adjoining the loop $\omega$ at the loop
corresponding to the top summand, which again corresponds to the
projective-injective $\Gamma_T$-module. Also, the relations for $\Lambda_T$
are the relations for $\widetilde{\Gamma}_T$ and in addition $\omega^2=0$.
\end{proof}

See Figure \ref{fig:example} for an example of a maximal rigid object $T$ and
the tilted algebra $\Gamma_T$, the cluster-tilted algebra
$\widetilde{\Gamma}_T$ and the endomorphism ring $\Lambda_T$.
\begin{figure}
\centering
\includegraphics[width=14cm]{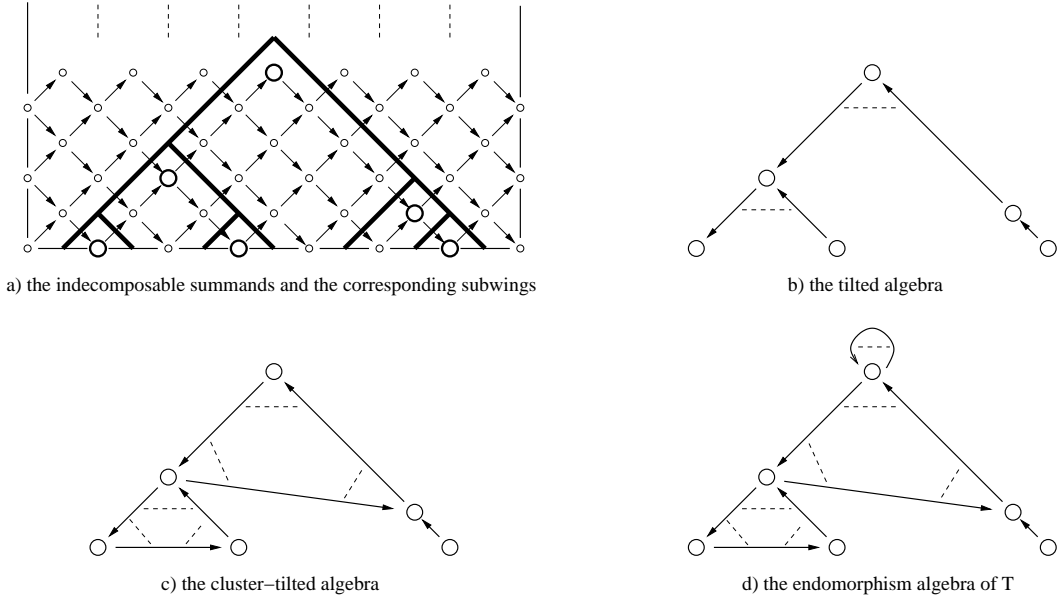}
\caption{An $n=7$ example of a maximal rigid object $T$ and the associated
  algebras $\Gamma_T$, $\widetilde{\Gamma}_T$ and
  $\Lambda_T$.} \label{fig:example}
\end{figure}

An explicit description of the quivers for cluster-tilted algebras of type $A$
was given in \cite{seven}, and also in \cite{bv}. It can be deduced from the
type $A$ cluster category model from \cite{ccs}. They are exactly the quivers
satisfying the following:
\begin{itemize}
\item all non-trivial minimal cycles are oriented and of length 3
\item any vertex has valency at most four
\item if a vertex has valency four, then two of its adjacent arrows belong to
  one 3-cycle, and the other two belong to another 3-cycle
\item if a vertex has valency three, then two of its adjacent arrows belong to
  a 3-cycle, and the third does not belong to any 3-cycle
\end{itemize}
In the first condition, a \emph{cycle} means a cycle in the underlying graph.
A \emph{connecting vertex} for such a quiver, as defined in \cite{v}, is a
vertex which either has valency one, or has valency two and is traversed by a
3-cycle.

Note that for the endomorphism ring $\Lambda_T$ of a maximal rigid $T$, the
loop vertex is connecting for the quiver of $\widetilde{\Gamma}_T$. 
There is a sort of converse to Theorem \ref{theorem:endoring}, so we have the
full description of the algebras which can arise:

\begin{prop}
Let $\widetilde{\Gamma}$ be a cluster-tilted algebra of type $A_{n-1}$, and
let $c$ be a connecting vertex for $Q_{\widetilde{\Gamma}}$. Then there exists
a maximal rigid object $T$ of $\C_n$ such that $Q_{\Lambda_T}$ is obtained
from $Q_{\widetilde{\Gamma}}$ by adjoining a loop at $c$.
\end{prop}

\begin{proof}
We do induction on $n$. The claim is easily verified for small values.

Given a cluster-tilted algebra $\widetilde{\Gamma}$ of type $A_{n-1}$, let $Q$
be its quiver and let $c$ be some connecting vertex of $Q$. Then $c$ has
valency 1 or 2 in $Q$. Consider first the case where $c$ has valency 2.
Then $c$ is traversed by a 3-cycle $c\to c_1\to c_2\to c$ in $Q$. The quiver
$Q\backslash \{ c,c_1\to c_2\}$ has two disconnected components $Q_1$ and $Q_2$,
where $c_i$ is connecting for $Q_i$. Also, $Q_i$ is the quiver of some
cluster-tilted algebra $\widetilde{\Gamma}_i$ of type $A_{k_i}$, where
$k_1+k_2=n-2$.

By induction we can assume that for each $i=1,2$ there exists a maximal rigid
object $T_i$ in $\C_{k_i+1}$ such that the endomorphism ring of $T_i$ is
obtained from the quiver of $\widetilde{\Gamma}_i$ by adjoining a loop at
$c_i$. Let $M_i$ be the corresponding tilting $k\vec{A}_{k_i}$-module. We have
a natural embedding of the module category $\mod k\vec{A}_{k_1}$ into the wing
$\W_{(1,k_1)}$, and of $\mod k\vec{A}_{k_2}$ into $\W_{(k_1+2,k_2)}$. It is
now easily seen that the images of the indecomposable summands of
$M_1$ and $M_2$ under these embeddings are all compatible, and that the direct
sum of these can be completed to a maximal rigid object $T$ by adding the
object $(1,n-1)$. Then the quiver of $T$ is obtained from $Q$ by adding a loop
at $c$.

The case where $c$ has valency one in $Q$ is easier; one considers
$Q\backslash \{ c,c\frac{\quad}{\quad}c'\}$ which is cluster-tilted of type
$A_{n-2}$, and uses induction similarly as above.
\end{proof}

\begin{rem}
There are in fact exactly $n$ maximal rigid objects in $\C_n$ with the
prescribed endomorphism algebra, and these form a $\tau$-orbit.
\end{rem}

\section{Gentleness, Gorenstein dimension and indecomposable
  modules} \label{sec:gentle}

In this section we show that the endomorphism rings under discussion are
gentle. We use this to determine their Gorenstein dimension.

A finite-dimensional algebra $kQ/I$ where $Q$ is a finite quiver is called
\emph{special biserial} \cite{skw} if
\begin{itemize}
\item[(i)] for all vertices $v$ in $Q$, there are at most two arrows starting
  in $v$ and at most two arrows ending in $v$
\item[(ii)] for every arrow $\beta$ in $Q$, there is at most one arrow
  $\alpha_1$ in $Q$ with $\beta \alpha_1 \not \in I$ and at most one arrow
  $\gamma_1$ with $\gamma_1 \beta \not \in I$.
\end{itemize}
A special biserial algebra $kQ/I$ is \emph{gentle} \cite{ask} if moreover
\begin{itemize}
\item[(iii)] $I$ is generated by paths of length 2
\item[(iv)] for every arrow $\beta$ in $Q$ there is at most one arrow
  $\alpha_2$ such that $\beta \alpha_2$ is a path and $\beta \alpha_2 \in I$,
  and at most one arrow $\gamma_2$ such that $\gamma_2 \beta$ is a path and
  $\gamma_2 \beta \in I$.
\end{itemize}

\begin{theorem} \label{theorem:gentle}
If $T$ is a maximal rigid object in the cluster tube $\C$, then the
endomorphism ring $\End_{\C}(T)\op$ is gentle.
\end{theorem}

\begin{proof}
It follows from the description of cluster-tilted algebras of type $A$, and
Theorem \ref{theorem:endoring}, that condition (i) is satisfied whenever $v$
is not the loop vertex. If $v$ is the loop vertex, the quiver generally looks
locally like this:
\[
\includegraphics[width=2.5cm]{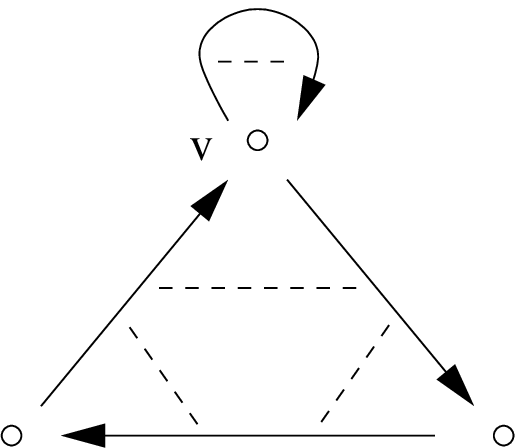}
\]
where the relations are indicated by dashed lines and it may also happen that
one of the two other vertices pictured and consequently the adjacent arrows
are not there. We see that also for this vertex (i) is fulfilled. 

As proved in \cite{bv}, any cluster-tilted algebra of type $A$ is gentle, and
therefore if $\beta$ is not an arrow incident with the loop vertex, (ii) and
(iv) are satisfied. If $\beta$ is incident with the loop vertex, (ii) and (iv)
follow from the local description pictured above, with the observation that by
the description of the cluster-tilted algebras there are no minimal relations
involving both an arrow in the picture and an arrow outside the picture.

Moreover, by Theorem \ref{theorem:endoring}, the ideal $I$ is generated
by paths of length two, so (iii) is fulfilled, and the algebra is gentle.
\end{proof}

It is known from \cite{gr} that all gentle algebras are Gorenstein, that is, a
gentle algebra $G$ has finite injective dimension as both a left and a right
$G$-module. This dimension is then called the Gorenstein dimension of $G$.

In order to prove the next result, we need to recall the main result in
\cite{gr} in more detail. Let $G=kQ/I$ be a gentle algebra. An arrow $\alpha$
in $Q$ is said to be \emph{gentle} if there is no arrow $\alpha_0$ such that
$\alpha \alpha_0$ is a non-zero element of $I$. A \emph{critical path} in $Q$
is a path $\alpha_t \cdots \alpha_2 \alpha_1$ such that
$\alpha_{i+1}\alpha_i\in I$ for all $i=1,...,n-1$.

\begin{theorem}[Gei\ss , Reiten \cite{gr}] \label{theorem:gr}
Let $G=kQ/I$ be a gentle algebra, and let $n(G)$ be the supremum of the
lengths of critical paths starting with a gentle arrow. ($n(G)$ is taken to be
zero if there are no gentle arrows.)
\begin{itemize}
\item[(a)] $n(G)$ is bounded by the number of arrows in $Q$.
\item[(b)] If $n(G)>0$, then $G$ is Gorenstein of Gorenstein dimension $n(G)$.
\item[(c)] If $n(G)=0$, then $G$ is Gorenstein of Gorenstein dimension at most
  one.
\end{itemize}
\end{theorem}

We can use this to find the Gorenstein dimension of our algebras:

\begin{prop} \label{prop:gorenstein}
Let $T$ be a maximal rigid object in $\C$. If $n=2$, the Gorenstein dimension
of $\Lambda_T = \End_{\C}(T)\op$ is zero, and if $n\geq 3$, the Gorenstein
dimension is one.
\end{prop}

\begin{proof}
By Theorem \ref{theorem:gentle}, the algebra $\Lambda_T$ is gentle, and we
can apply Theorem \ref{theorem:gr}.

If $n=2$, then $T$ has only one summand, and the endomorphism algebra
$\Lambda_T$ is isomorphic to the self-injective algebra $k[x]/(x^2)$, so in
this case the Gorenstein dimension is zero.

Assume therefore that $n\geq 3$. The gentle arrows are exactly the arrows
which are not traversed by any minimal oriented cycle. (In particular, the
loop is not gentle.) But if $\alpha$ is any such arrow, then $\beta \alpha$ is
a path for at most one arrow $\beta$, and $\beta \alpha$ can never be a zero
relation, since it is not a part of a 3-cycle. So if gentle arrows exist, the
maximal length $n(\Lambda_T)$ of critical paths starting in gentle arrows is
1, and therefore the Gorenstein dimension of $\Lambda_T$ is also 1. If gentle
arrows do not exist, we have $n(\Lambda_T)=0$, and the Gorenstein dimension is
at most 1.

It remains to be shown that $\Lambda_T$ cannot be self-injective for $n\geq
3$. For this, consider the indecomposable projective associated with the loop
vertex. It is easily seen that it is not injective.
\end{proof}

Since $\Lambda_T=\End_{\C}(T)\op$ is gentle, it is in particular a string
algebra. We will use this to show that $\Lambda_T$ has finite representation
type. For this, we will recall some basic facts about representations of string
algebras. More details on this can be found e.g. in \cite{bri}.

Let $kQ/I$ be a string algebra. We consider words from the alphabet formed by
arrows in $Q$ and their formal inverses. Inverse words are defined in the
obvious way. For an inverted arrow $\alpha^{-1}$ we set the end vertex
$e(\alpha^{-1})$ equal to the start vertex $s(\alpha)$ of the original arrow
and vice versa. A word $w = \alpha_t \cdots \alpha_2\alpha_1$, where
no two consecutive $\alpha_i$ are inverses of each other, is called a
\emph{string} if $e(\alpha_i)=s(\alpha_{i+1})$ for $i=1,...,t-1$ and moreover
no subword of $w$ or its inverse is a zero relation. In addition, there is a
trivial string of length zero for each vertex of $Q$. We say that the start-
and end vertices of the $\alpha_i$ which appear in the strings are the
vertices \emph{traversed} by the string. For technical purposes, we also
consider a unique zero (or empty) string of length $-1$.

To any string $\sigma=\alpha_t...\alpha_2\alpha_1$ of length $t$ in $Q$ there
is an associated indecomposable $(t+1)$-dimensional
representation $M(\sigma)$ of $kQ/I$ given by one-dimensional vector spaces in
each vertex traversed by the string (with multiplicity) and one-dimensional
identity maps for each of the arrows (and inverted arrows) appearing in
$\sigma$. We have that $M(\sigma_1)\simeq M(\sigma_2)$ if and only if
$\sigma_1=\sigma_2$ or $\sigma_1=\sigma_2^{-1}$. The $kQ/I$-modules given by
such representations are called \emph{string modules}. If there exist closed
strings $\alpha_t \cdots \alpha_2\alpha_1$ (i.e. strings starting and ending
in the same vertex) such that $\alpha_t \not = \alpha_1^{-1}$ and powers of
the string are strings as well, there will also be infinite families of
indecomposables called \emph{band modules}. The string modules and band
modules constitute a complete set of representatives of isoclasses of
indecomposable modules over $kQ/I$.

Now some remarks on the strings in the quiver of $\Lambda_T$, which will be
useful in the proof of Theorem \ref{thm:finitetype} and in Section
\ref{sec:homfunctor}. In what follows, an arrow $\alpha$ that arises from an
irreducible $\T$-map in $\add_{\C} T$ will be called a $\T$-arrow, and
similarly an arrow that arises from a $\D$-map will be called a $\D$-arrow.

Since the vertices of the quiver are in a natural bijection with the
indecomposable summands of $T$, we will transfer some terminology about the
summands to the vertices. So the quasilength of a vertex is the quasilength
of the corresponding summand, and vertices (also arrows, strings) are said to
be in a wing if the corresponding summands are in the wing.

\begin{lem} \label{lem:stringsandcycles}
Let $(T_i;T_j,T_k)$ be a non-degenerate $T$-subwing triple, and
$\alpha_{ij},\alpha_{ki}$ and $\beta_{jk}$ the corresponding arrows. Then
\begin{itemize}
\item[(i)] If $\sigma_2\alpha_{ki}\sigma_1$ is a string, then $\sigma_1$ is
  not of the form $\sigma_1=\beta_{jk}\sigma^*_1$.
\item[(ii)] If $\sigma_2\alpha_{ij}\sigma_1$ is a string, then $\sigma_2$ is
  not of the form $\sigma_2=\sigma^*_2\beta_{jk}$.
\end{itemize}
The analogous statements hold for the inverses of the arrows.
\end{lem}

\begin{proof}
The assertions follow from the fact that $\alpha_{ki}\beta_{jk}$ and
$\beta_{jk}\alpha_{ij}$ are both zero relations.
\end{proof}

\begin{lem}\quad
\begin{itemize}
\item[(i)] If $\beta:i\to j$ is a $\D$-arrow, and $\sigma_2\beta\sigma_1$ is a
  string, then $\sigma_2$ is in the wing of $j$ and $\sigma_1$ is in the wing
  of $i$, and similarly for inverses of $\D$-arrows.
\item[(ii)] A string in the quiver of $\Lambda_T$ contains at most one
  $\D$-arrow or inverse of such.
\end{itemize}
\end{lem}

\begin{proof}
\textbf{(i):} Let $\sigma_2\beta\sigma_1$ be a string, where $\beta:i\to j$ is
a $\D$-arrow corresponding to a $\D$-map $\psi_{ji}:T_j \to T_i$.

Suppose first that $T_i\not =T_j$. Then there is a subwing triple
$(T_k;T_i,T_j)$. Now if the string $\sigma_2$ which starts in $j$ is the
trivial string for vertex $j$, then there is nothing to show, since $j$ is
definitely in the wing of itself. Assume therefore that $\sigma_2$ has length
at least one. Then $\sigma_2= a_k...a_1a_0$, where $a_0$ is either an arrow
starting in $j$ or the inverse of an arrow ending in $j$. By Lemma
\ref{lem:stringsandcycles} we know that $a_0$ cannot be the arrow
$\alpha_{jk}:j\to k$ which connects $j$ to the vertex $k$ associated with
$T_k$. Now the remaining possibilities for $a_0$ are contained in the wing of
$j$. By Lemma \ref{lem:stringsandcycles} again, it follows that none of the
$a_i$ can be $\D$-arrows (or inverse $\D$-arrows). So $\sigma_2$ traverses
vertices of successively smaller quasilength and cannot return to $j$. So
$\sigma_2$ is in the wing of $j$.

If $T_i =T_j$, then $T_i$ is the top summand and $\beta$ is the loop, in which
case the claim follows by the fact that $\beta^2$ is a zero relation, and a
similar argument as above.

The statement for $\sigma_1$ and $i$ is proved analogously, and also the
statements for inverses of $\D$-arrows.

\textbf{(ii):} This follows from (i). If $i\overset{\beta_1}{\to} j$ and
$k\overset{\beta_2}{\to}l$ are two $\D$-arrows, then clearly $\beta_2$ cannot
be in the wing of $i$ or $j$ when at the same time $\beta_1$ is in the wing of
$k$ or $l$. So they cannot both appear in the same string. The same goes if
one (or both) of $\beta_1$ and $\beta_2$ is the inverse of a $\D$-arrow.
\end{proof}

\begin{rem} \label{rem:closed}
It follows from this lemma that if $\sigma$ is a closed string, then either
$\sigma$ is a trivial string, or $\sigma$ contains the loop as the only
$\D$-arrow.
\end{rem}

\begin{lem} \label{lem:stringsubwing}
In the following, we consider strings only up to orientation.
\begin{itemize}
\item[(i)] The strings of length $k-1$ in the quiver of $\Lambda_T$ which do
  not contain a $\D$-arrow or inverse of such are in bijection with sequences
  $T_{i_1},...,T_{i_k}$ such that there are subwing triples
  $(T_{i_j};T_{i_{j+1}},T^*_{i_{j+1}})$ or $(T_{i_j};T^*_{i_{j+1}},T_{i_{j+1}})$ for
  $j=1,...,k-1$.
\item[(ii)] The strings in the quiver of $\Lambda_T$ that do not contain a
  $\D$-arrow or inverse of such are in bijection with pairs $T_i,T_j$ of
  summands of $T$ such that $T_i\in \W_{T_j}$.
\end{itemize}
\end{lem}

\begin{proof}
\textbf{(i):} Let $\sigma$ be a string without a $\D$-arrow or inverse
$\D$-arrow. If $\sigma$ is trivial, the claim is obviously true, so assume
$\sigma$ has length $\geq 1$. Choose $T_{i_1}$ to be an indecomposable summand
of $T$ which corresponds to a vertex $i_1$ traversed by $\sigma$ such that no
of the other vertices traversed by $\sigma$ have higher quasilength. Assume
$\ql T_{i_1}>1$. Then there is some $T$-subwing triple
$(T_{i_1};T'_{i_1},T''_{i_1})$. Let $\alpha'':i''_1 \to i_1$ and $\alpha':i_1
\to i'_1$ be the corresponding arrows.

Now since there are no $\D$-arrows, and $i_1$ has maximal quasilength among
the vertices traversed by $\sigma$, the string $\sigma$ must be of the form
$\sigma = \sigma_2 \sigma_1$ where $\sigma_1$, if it is non-trivial, is of the
form $\sigma_1 = \alpha'' \sigma^*_1$ or $\sigma_1 =
(\alpha')^{-1}\sigma^*_1$, and similarly $\sigma_2 = \sigma^*_2 \alpha'$ or
$\sigma_2 = \sigma^*_2 (\alpha'')^{-1}$ if it is non-trivial. So for $\sigma_2
\sigma_1$ to be a string, one of these has to be the trivial string associated
with $i_1$, since the composition $\alpha' \alpha''$ is zero.

So $i_1$ is the start- or end vertex of $\sigma$, and the first arrow (or
inverse arrow) connects $i_1$ to one of the vertices from the $T$-subwing
triple with $i_1$ on top. By repeating the process with the string
$\sigma^*_1$ or $\sigma^*_2$, we get the desired chain of subwing triples.

\textbf{(ii):} Follows directly from (i).
\end{proof}

See Figure \ref{fig:stringexample} for an example of strings in the quiver of
$\Lambda_T$.
\begin{figure}
\centering
\psfrag{omega}{$\omega$}
\includegraphics[width=11cm]{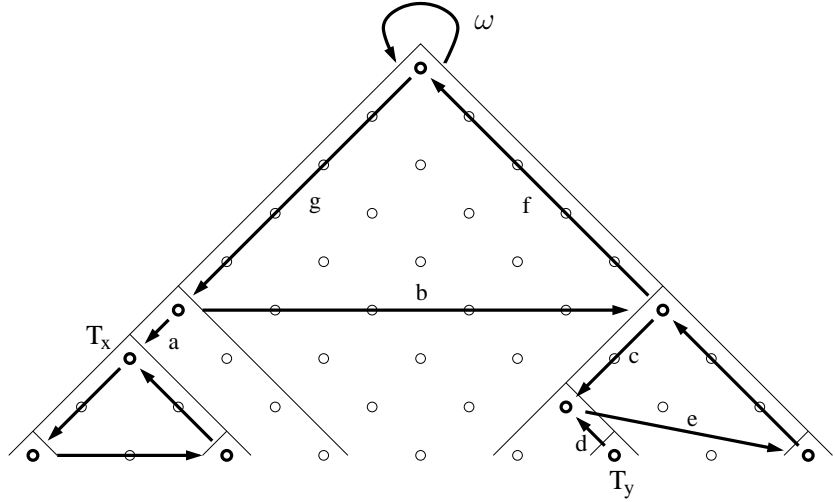}
\caption{The string $d^{-1}cba^{-1}$, where $b$ is the only $\D$-arrow, starts
  in the vertex corresponding to $T_x$ and ends in the vertex corresponding to
  $T_y$. Note how a path like $ec$ (or its inverse) can not be in a string,
  since it is a zero relation. Also, the path $gf$ is not a string, while
  both $g\omega f$ and $g\omega^{-1} f$ are.} \label{fig:stringexample}
\end{figure}

\begin{theorem} \label{thm:finitetype}
For a maximal rigid object $T$ in $\C_n$, the endomorphism ring $\Lambda_T$ is
of finite type, and the number of indecomposable representations is
\[
\frac{1}{2}(3n^2-5n+2).
\]
\end{theorem}

\begin{proof}
Let $Q$ be the quiver of $\Lambda_T$ and $I$ the relation ideal. Moreover, let
$l$ denote the loop vertex, and $\omega$ the loop itself. For a string
$\sigma$ we will denote the associated indecomposable representation by
$M(\sigma)$. 

First consider strings which do not involve the loop. These are in bijection
with ordered pairs of vertices: For each pair $i,j$ from $Q_0$, let $T_i,T_j$
be the associated summands of $T$. If $T_i\in \W_{T_j}$ or vice versa, there
is, as in Lemma \ref{lem:stringsubwing}, a string without a $\D$-arrow
connecting $i$ and $j$. So suppose this is not the case. Consider the
(unique) summand $T_k$ which is of minimal quasilength such that $T_i$ and
$T_j$ are both in $\W_{T_k}$. Then there is some $T$-subwing triple
$(T_k;T'_i,T'_j)$ where $T_i\in \W_{T'_i}$ and $T_j\in \W_{T'_j}$ or the other
way around. Now use Lemma \ref{lem:stringsubwing} again to find strings
connecting $i$ with the vertex $i'$ associated with $T'_i$ and $j$ with the
vertex $j'$ associated with $T'_j$. Now there is a $\D$-arrow
$\beta'_{ij}:i'\to j'$. Connecting the two string by way of $\beta'_{ij}$
yields the desired string, and we can choose the orientation, so this is up to
ordered pairs of vertices.

In particular, by Remark \ref{rem:closed}, any non-loop string starting and
ending in the same vertex is a trivial string.

Denote by $\sigma (i,j)$ the unique non-loop string starting in $i$ and ending
in $j$, so $\sigma(i,j)^{-1}=\sigma(j,i)$. For the corresponding
indecomposable representations we have isomorphisms $M(\sigma (i,j))\simeq
M(\sigma (j,i))$. The simple representations are the $M(\sigma(i,i))$. The
total number of representations corresponding to non-loop strings is therefore
\[
(n-1) + {n-1 \choose 2} = \frac{1}{2}(n^2-n)
\]
where the first term is the number of simple representations and the second
is the number of strings of length $\geq 1$ up to orientation.

Now for the strings passing through the loop. For each pair $i,j$ of vertices,
there are two strings from $i$ to $j$ passing through the loop $\omega$:
\begin{eqnarray*}
\sigma_{\omega}(i,j) & = & \sigma(l,j)\omega\sigma(i,l) \\
\sigma^-_{\omega}(i,j) & = &  \sigma(l,j)\omega^{-1}\sigma(i,l)
\end{eqnarray*}
and these are all possible loop strings. We see that
$(\sigma_{\omega}(i,j))^{-1}=\sigma^-_{\omega}(j,i)$, so
$M(\sigma_{\omega}(i,j))\simeq M(\sigma^-_{\omega}(j,i))$ for any choice of
$i,j$, in particular for $i=j$. We deduce that the indecomposable
representations associated with loop strings are in bijection with ordered
pairs of vertices, and the number of such representations is $(n-1)^2$.

These are all the strings there are. Since (by Remark \ref{rem:closed} and
Lemma \ref{lem:stringsubwing}) the only closed strings are either
trivial, $\omega$ or of the form $\alpha \sigma \alpha^{-1}$ for some arrow
(or inverse arrow) $\alpha$ and some string $\sigma$, it follows that there
are no band modules. So the string modules we have presented form a complete
set of isomorphism classes of $\Lambda_T$-modules.

Summarising, we find that the total number of representations is
\[
\frac{1}{2}(n^2-n)+(n-1)^2=\frac{1}{2}(3n^2-5n+2).
\]
\end{proof}

\section{On the behaviour of the Hom-functor} \label{sec:homfunctor}

For a cluster-tilted algebra $C_T = \End_{\C_H}(T)\op$ arising from the cluster
category $\C_H$ of some hereditary algebra $H$, there is a close connection
between the module category of $C_T$ and the cluster category itself. The main
theorem from \cite{bmr} says that the functor $G = \Hom_{\C_H}(T,-):\C_H \to
\mod C_T$ induces an equivalence
\[
\bar{G} : \C_H/\add \tau T  \overset{\sim}{\longrightarrow} \mod C_T
\]
In particular, the cluster-tilted algebra is of finite representation type if
and only if $\C_H$ has finitely many objects (which again happens if and only
if $H$ is of finite type). By Theorem \ref{thm:finitetype}, a similar result
cannot hold for the cluster tubes. The analogous argument fails because the
Hom-functor is not full. In this section, we will study some properties of
this functor.

We introduce some notation. For any indecomposable object $X$ in $\C$, let
$H(X)=H^{\T}(X)\cup H^{\D}(X)$ be the Hom-hammock of $X$, where $H^{\T}(X)$ is
the set of indecomposables to which $X$ has $\T$-maps, and similarly for
$H^{\D}(X)$. Also, consider the \emph{reverse
  Hom-hammock} $R(X)\subset \ind \C$, that is, the support of
$\Hom_{\C}(-,X)$ among the indecomposables. Like the ordinary Hom-hammock,
this has a natural structure as the union of two components, one denoted by
$R^{\T}(X)$ containing the indecomposables that have non-zero $\T$-maps to
$X$, and another one denoted by $R^{\D}(X)$ containing those that have
$\D$-maps to $X$. Note that by the description of the Hom-hammocks,
$R^{\T}(X)=H^{\D}(\tau^{-2}X)$ and $R^{\D}(X)=H^{\T}(\tau^{-2}X)$. So the
shape of $R(X)$ is similar to the shape of $H(X)$ (Figure
\ref{figure:homhammock}).

For the remainder of this chapter, $T = \amalg_{i=1}^{n-1} T_i$ will be a
maximal rigid object in $\C_n$, and we assume that the top summand of $T$ is
$T_1=(1,n-1)$. Clearly, by redefining the coordinates we can use the results
for all maximal rigid objects of $\C$. As in the preceeding sections, we will
denote by $\Lambda_T$ the endomorphism ring $\Lambda_T=\End_{\C}(T)\op$.

We define $\F$ to be a certain set of indecomposable objects in $\C$:
\[
\F = \left\{ X = (a,b)\; |\; b \leq n-1 \right\} \cup
\left\{ X = (a,b)\; |\; a+b\leq 2n-1 \right\}
\]
See Figure \ref{fig:fundamentaldomain}. The region $\F$ in the tube consists
of the rigid part and in addition a triangle of height $n-1$ in the non-rigid
part. (We have defined wings only for rigid indecomposables, but we can think
of $\F$ as the wing of the object $(1,2n-2)$.)
\begin{figure}
\centering
\includegraphics[width=7cm]{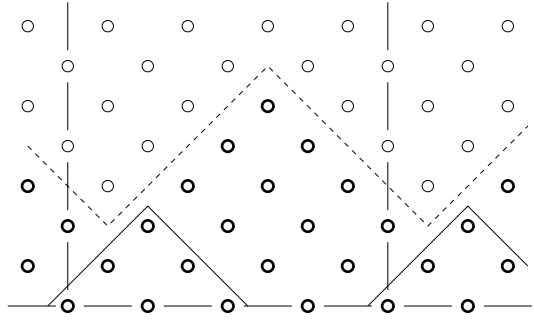}
\caption{The set $\F$ in $\C_4$, below the dashed curve. $T$ is concentrated in
the indicated wing.}
\label{fig:fundamentaldomain}
\end{figure}

The following claims are easily verified:

\begin{lem} \label{lem:onecomponent}
If $X$ is an indecomposable in $\W_{T_1}$, then $H^{\T}(X)\cap \F$ forms one
rectangle-shaped subgraph of the tube, and similarly for $H^{\D}(X)\cap \F$.
\end{lem}

\begin{lem} \label{lem:revhammquasi}
Let $*$ be either $\T$ or $\D$. Then for an indecomposable $X$, the set
$R^*(X)$ contains a unique quasisimple $q^*_X$, and a necessary condition for
an object $Y$ to be in $R^*(X)$ is that $q^*_X\in \W_Y$.
\end{lem}

\begin{lem} \label{lem:T1inR(X)}
Let $X\in \F$. Then $T_1\not \in R(X)$ if and only if $X\in \W_{\tau T_1}$.
\end{lem}

We now want to assign to each indecomposable in $\F$ a uniquely defined
string in the quiver of $\Lambda_T$. In the main result of this section we
will show that the images under the Hom-functor are given by these
strings. The first step is to encode information about $\T$-maps and $\D$-maps
in separate strings, which will be joined to one string at a later stage.

\begin{lem} \label{lem:revhamm}
Let $*$ be either $\T$ or $\D$. For $X\in \F$ we have the following.
\begin{itemize}
\item[(i)] $R(X)\cap \add T$ is empty if and only if $X\in \add \tau T$.
\item[(ii)] For any $T$-subwing triple $(T_i;T_j,T_k)$, at most one of $T_j$
  and $T_k$ can be in $R^*(X)$.
\item[(iii)] If $R^*(X)\cap \add T$ is non-empty, there is a unique string
  in the quiver of $\Lambda_T$ traversing each of the vertices
  corresponding to the indecomposables in $R^*(X)\cap \add T$ exactly once
  (and no other vertex) and ending in the vertex corresponding to the summand in
  $R^*(X)\cap \add T$ of highest quasilength.
\item[(iv)] A string of the type in part (iii) contains no $\D$-arrow (or
  inverse of such).
\end{itemize}
\end{lem}

\begin{proof}
\textbf{(i):} We need to show that $\Hom_{\C}(T,X)=0$ if and only if $X\in
\add \tau T$. If $X\in \add \tau T$, then there are no non-zero maps from $T$
to $X$ since $T$ is rigid.

For the converse, assume that the intersection is empty. Using Lemma
\ref{lem:T1inR(X)}, we get that $\ql X\leq n-1$. Moreover, we have that
\[
\Ext^1_{\C}(T,\tau^{-1}X) = \Hom_{\C}(T,X) = 0.
\]
Since $\ql X \leq n-1$, the object $X$, and consequently $\tau^{-1}X$, is
rigid. So $\tau^{-1}X = T_i$ for some $i$ since $T$ is maximal rigid, and we
can conclude that $X \in \add \tau T$.

\textbf{(ii):} We know that $\W_{T_j}$ and $\W_{T_k}$ are disjoint. The claim
then follows from Lemma \ref{lem:revhammquasi}.

\textbf{(iii):} Assume $R^*(X)\cap \add T$ is non-empty, and let $T_l$ and
$T_h$ be elements in this set with minimal and maximal quasilength,
respectively. By Lemma \ref{lem:revhammquasi}, the unique quasisimple $q^*_X$
which is in $R^*(X)$ is now in both $\W_{T_l}$ and $\W_{T_h}$. So in
particular $\W_{T_l}$ and $\W_{T_h}$ have non-empty intersection, and
therefore by Lemma \ref{lem:subwingintersection} we know that $T_l\in
\W_{T_h}$. Also by Lemma \ref{lem:subwingintersection} we see that $T_h$ and
$T_l$ are uniquely determined. There is some $T$-subwing triple
$(T_l;T'_l,T''_l)$, and it can easily be seen that if $q^*_X$ were in
$\W_{T'_l}$, say, then also $T'_l$ would be in $R^*(X)$, which would violate
the minimality condition on $T_l$. Thus $T_l$ is the summand of smallest
quasilength such that $q^*_X$ is in the corresponding wing. See Figure
\ref{fig:revhammandsubwings}.
\begin{figure}
\centering
\includegraphics[width=8cm]{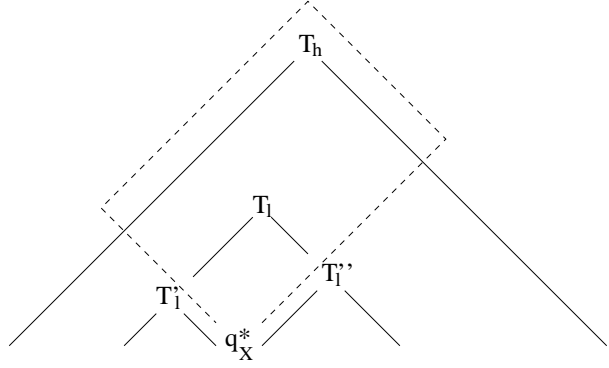}
\caption{If $T_l$ and $T_h$ are the summands of $T$ in $R^*(X)$ of lowest and
  highest quasilength, respectively, then the entire rectangle indicated must
  be in $R^*(X)$.} \label{fig:revhammandsubwings}
\end{figure}

Also, for summands $T_s$ we have that $T_s\in R^*(X)$ if and only if $\W_{T_l}
\subseteq \W_{T_s} \subseteq \W_{T_h}$, again by Lemma
\ref{lem:subwingintersection} and the maximality of $T_h$. Now the desired
string is a string of the type described in Lemma \ref{lem:stringsubwing},
oriented in the suitable way.

\textbf{(iv):} By part (ii), a $\D$-arrow associated with a non-degenerate
$T$-subwing triple $(T_i;T_j,T_k)$ could not be traversed by a string of the
type described in part (iii). Moreover, the loop is disallowed as well, since
then the loop vertex would be traversed twice, contrary to the condition in
part (iii).
\end{proof}

For an indecomposable $X\in \F$ such that $R^*(X)\cap \add T$ is non-empty,
where $*$ is either $\T$ or $\D$, we denote the string in Lemma
\ref{lem:revhamm} part (iii) by $\sigma^*_X$. If the intersection is empty, we
define $\sigma^*_X$ to be the zero string. The two next lemmas tell us
that different objects in $\F \backslash \add T$ can be distinguished by their
associated strings.

\begin{lem} \label{lem:differentstrings}
Let $X, Y \in \F \backslash \add \tau T$. If
$\sigma^{\T}_X =\sigma^{\T}_Y$ and $\sigma^{\D}_X =\sigma^{\D}_Y$, then $X=Y$.
\end{lem}

\begin{proof}
We note first that for any indecomposable $Z$, the unique quasisimple
$q^{\T}_Z$ in $R^{\T}(Z)$ determines the first coordinate of $Z$. 
Similarly, the unique quasisimple $q^{\D}_Z$ in $R^{\D}(Z)$ determines the
sum of the coordinates of $Z$ modulo the rank $n$. So if the first coordinate
of $Z$ is known, the quasisimple $q^{\D}_Z$ determines the second coordinate
\emph{modulo $n$}.

Let now $X$ and $Y$ be in $\F \backslash \add T$ such that $\sigma^{\T}_X
=\sigma^{\T}_Y$ and $\sigma^{\D}_X =\sigma^{\D}_Y$. We aim to show that $X$
and $Y$ must be equal.

By Lemma \ref{lem:revhamm} part (i), at least one of $\sigma^{\T}_X
=\sigma^{\T}_Y$ and $\sigma^{\D}_X =\sigma^{\D}_Y$ is non-zero. Assume first
that both are non-zero. We claim that $q^{\T}_X=q^{\T}_Y$ and
$q^{\D}_X=q^{\D}_Y$. As in the proof of Lemma \ref{lem:revhamm} part (iii), we
observe that if $T_l$ is the $T$-summand of smallest quasilength in $R^{\T}(X)$,
then $q^{\T}_X$ is the unique quasisimple which is in $\W_{T_l}$ but not in
the wing of any $T$-summand of smaller quasilength. Since $T_l$ is also the
summand of smallest quasilength in $R^{\T}(Y)$, we must have
$q^{\T}_X=q^{\T}_Y$. Similarly, we deduce that $q^{\D}_X=q^{\D}_Y$.

Thus $X$ and $Y$ have the same first coordinate, and the same second
coordinate modulo $n$. But since $X$ and $Y$ are in $\F$, this means that
unless $X$ and $Y$ are equal, one of them is in $\W_{(1,n-2)}$ and the other
is in the non-rigid part. If they are not equal, there is then a contradiction
to Lemma \ref{lem:T1inR(X)}: If $X$ is in $\W_{(1,n-2)}$ and $Y$ is in the
non-rigid part, then by Lemma \ref{lem:T1inR(X)} $T_1\in R(Y)$ but $T_1 \not
\in R(X)$, which is impossible since we have assumed
$\sigma^{\T}_X=\sigma^{\T}_Y$ and $\sigma^{\D}_X=\sigma^{\D}_Y$. We conclude
that if both $\sigma^{\T}_X=\sigma^{\T}_Y\not =0$ and
$\sigma^{\D}_X=\sigma^{\D}_Y\not =0$, then $X=Y$.

Assume then that $\sigma^{\T}_X=\sigma^{\T}_Y\not =0$ and $\sigma^{\D}_X
=\sigma^{\D}_Y=0$, and furthermore that $X\not =Y$. Then $X$ and $Y$ have the
same first coordinate. Moreover, at least one of $X$ and $Y$ must be in
$\W_{(1,n-2)}$, since the only other possible positions for an object $Z\in
\F$ such that $\sigma^{\T}_Z\not =0$ and $\sigma^{\D}_Z =0$ are on the coray
$\mathbf{C}_{(n-1,n)}$.

Suppose (without loss of generality) that $X$ has smaller quasilength than
$Y$. So in particular, $X$ is in $\W_{(1,n-2)}$. Let $T_h$ be the object in
$R^{\T}(X)\cap \add T$ with highest quasilength. Since $X\in \W_{(1,n-2)}$,
we have that $T_h\not =T_1$. Therefore, there is some $T$-subwing triple
$(T_a;T_h,T'_h)$, where $T_h$ is necessarily on the left side since
$R^{\T}(X)$ contains whole corays, and so by the maximality of $T_h$,
there can be no more summands of $T$ on $\mathbf{C}_{T_h}$.

Assume first that this triple is non-degenerate. Since $R^{\D}(X)$ does
not contain any summands of $T$, there is in particular no $\D$-map $T'_h\to
X$. So $X$ is in $H^{\T}(T_h)$ but not in $H^{\D}(T'_h)$. Moreover, we see
that $X\not \in H^{\T}(T_a)$, by the maximality of $T_h$. 
So $X$ must be on the coray $\mathbf{C}_{\tau T'_h}$. If the triple is
degenerate, then $X$ must be on the right edge of $\W_{T_h}$, since $T_a\not
\in R^{\T}(X)$. In any of these two cases, we get a contradiction: Since $Y$
and $X$ have the same first coordinate, and $Y$ has higher quasilength, $T_a$
must be in $R^{\T}(Y)$. This contradicts the equality of $\sigma^{\T}_X$ and
$\sigma^{\T}_Y$, and our assumption that $X\not =Y$ must be wrong. See Figure
\ref{fig:equalsigmas}.
\begin{figure}
\centering
\includegraphics[width=9cm]{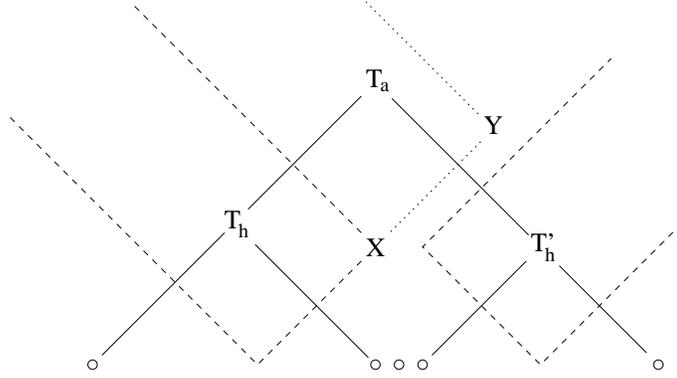}
\caption{If $\sigma^{\T}_X=\sigma^{\T}_Y\not =0$ and
  $\sigma^{\D}_X=\sigma^{\D}_Y=0$, then $X$ and $Y$ have the same first
  coordinate. If $X$ is inside some non-degenerate $T$-subwing triple, then
  $X$ must be on the coray $\mathbf{C}_{\tau T'_h}$, since otherwise $T'_h\in
  R^{\D}(X)$.} \label{fig:equalsigmas}
\end{figure}

The situation where $\sigma^{\T}_X=\sigma^{\T}_Y=0$ and
$\sigma^{\D}_X=\sigma^{\D}_Y\not =0$ can be proved in a similar manner.
\end{proof}

The following lemma is used to show that if two different objects have exactly
one associated $\sigma^{\T}$- or $\sigma^{\D}$-string, then the strings are
different.

\begin{lem} \label{lem:TstringisnotDstring}
If $\sigma^{\T}_X=\sigma^{\D}_Y\not =0$, then $\sigma^{\T}_Y$ is non-zero.
\end{lem}

\begin{proof}
Suppose that $\sigma^{\T}_X=\sigma^{\D}_Y\not =0$, and let $T_i$ be the
summand in $R^{\T}(X)\cap \add T=R^{\D}(Y)\cap \add T$ which has highest
quasilength. We claim that $T_i=T_1$. To see this, assume that $\ql
T_i<n-1$. Then there is some (degenerate or non-degenerate) $T$-subwing triple
$(T_k;T_i,T^*_i)$ or $(T_k;T^*_i,T_i)$. Since $R^{\T}(X)$ contains whole
corays, and $R^{\D}(Y)$ contains whole rays, the summand $T_k$ must be in one
of these two reverse Hom-hammocks. But this contradicts our choice of
$T_i$. So $T_i=T_1$.

Now observe that if $Y$ is in $\F$, and there is a $D$-map $T_1\to Y$, then
there is also a $\T$-map $T_1\to Y$, so in particular $\sigma^{\T}_Y\not =0$.
\end{proof}

We now show that if both strings associated with an indecomposable in $\F$ are
non-zero, then there is a larger string containing both of them.

\begin{lem} \label{lem:Dstrings}
Let $X\in \F$. If both $\sigma^{\T}_X$ and $\sigma^{\D}_X$ are non-zero, there
is a $\D$-arrow $\beta_X$ from the end vertex of $\sigma^{\T}_X$ to the
end vertex of $\sigma^{\D}_X$. So in particular, $(\sigma^{\D}_X)^{-1} \beta_X
\sigma^{\T}_X$ is a well-defined string.
\end{lem}

\begin{proof}
We consider four cases, depending on the position of $X$ in $\F$.

The first case is when there is a $\D$-map $T_1\to X$. One readily verifies
that there is then also a $\T$-map $T_1\to X$, so in this case $T_1$ is in
both $R^{\T}(X)$ and $R^{\D}(X)$, and the claim holds with the loop as
$\beta_X$.

The second case is when there is a $\T$-map $T_1\to X$, but no $\D$-map from
$T_1$ to $X$. This happens exactly when $X$ is on the coray
$\mathbf{C}_{(n-1,n)}$, and in this case there are no $\D$-maps from any
summands of $T$ to $X$, so $R^{\D}(X)\cap \add T$ is empty, and there is
nothing to prove.

The third case is when $X$ is located on the ray $\mathbf{R}_{(n,1)}$. Then
there are no $\T$-maps from $T$ to $X$, so again there is nothing to prove.

The only remaining situation is then when $X$ is in the wing
$\W_{(1,n-2)}$. As in the proof of Lemma \ref{lem:differentstrings}, let $T_h$
be the summand in $\add T\cap R^{\T}(X)$ of highest quasilength.
Since $T_h\not =T_1$, there is some $T$-subwing triple
$(T_a;T_h,T'_h)$ with $T_h$ necessarily on the left, since $R^{\T}(X)$
contains whole corays. Assume first that this triple is non-degenerate. Then,
since $T_h\in R^{\T}(X)$ and $T_a\not \in R^{\T}(X)$ by the maximality
property of $T_h$, we note the following about the position of $X$:
\begin{itemize}
\item[-] $X\in \W_{T_a}$, but not on the right edge of $\W_{T_a}$, since then
  there would be a $\T$-map $T_a\to X$
\item[-] $X\not \in \W_Y$, where $Y$ is the object on $\mathbf{C}_{T'_h}$
  which has an irreducible map to $T'_h$, since then there would be no
  $\T$-map $T_h\to X$
\item[-] if $X\in \W_{T_h}$, then it is on the right edge, since otherwise
  there would be no $\T$-map $T_h\to X$
\end{itemize}
Our aim is to show that $T'_h$ is in $R^{\D}(X)$, and moreover that it is the
summand of $T$ with highest quasilength appearing in $R^{\D}(X)$.

With the above remarks about the position of $X$, we see that the only allowed
positions such that there is no $\D$-map $T'_h\to X$ are positions on the
coray $\mathbf{C}_{\tau T'_h}$. But if $X$ were on this coray, then
$R^{\D}(X)\cap \add T$ would be empty, contrary to our hypothesis: Namely,
assuming this position for $X$, suppose there was some summand $T_b\in
R^{\D}(X)$. An equivalent condition to this (cf. Lemma \ref{lem:homspaces}) is
that there is a non-zero $\T$-map $X\to \tau^2 T_b$, which again is equivalent
to the existence of a $\T$-map $\tau^{-1}X \to \tau T_b$. But since $X\in
\mathbf{C}_{\tau T'_h}\cap \W_{T_a}$, we have that $\tau^{-1}X$ is on the
right edge of $\W_{T_a}$. So there would be a non-zero $\T$-map $T_a\to \tau
T_b$, which is impossible since $T_a$ and $T_b$ are Ext-orthogonal.

So $T'_h\in R^{\D}(X)$. Let $T_c$ be the $T$-summand of highest quasilength
which appears in $R^{\D}(X)$. Then, since both $\W_{T'_h}$ and $\W_{T_c}$
contain the quasisimple $q^{\D}_X$ from Lemma \ref{lem:revhammquasi}, Lemma
\ref{lem:subwingintersection} tells us that $T'_h\in \W_{T_c}$. But $T'_h$ is
in $\W_{T_a}$ as well, so by Lemma \ref{lem:subwingintersection} again, either
$T_a\in \W_{T_c}$ or $T_c\in \W_{T_a}$. The former case is not possible, since
it would imply that $T_a\in R^{\D}(X)$, which is impossible since $X\in
\W_{T_a}$. So the remaining possibility is that $T_c=T'_h$, that is, $T'_h$ is
the $T$-summand in $R^{\D}(X)$ of highest quasilength.

By the description of the quiver of $\Lambda_T$ in Section \ref{sec:endoring},
there is a $\D$-arrow $\beta_X$ associated with the non-degenerate $T$-subwing
triple, from the vertex corresponding to $T_h$ to the vertex corresponding to
$T'_h$. Since $\sigma^{\T}_X$ ends in the vertex corresponding to $T_h$, and
$(\sigma^{\D}_X)^{-1}$ starts in the vertex corresponding to $T'_h$, the
string $(\sigma^{\D}_X)^{-1} \beta_X \sigma^{\T}_X$ is well-defined.

It remains to consider the case where the $T$-subwing triple $(T_a;T_h,T'_h)$
is degenerate, that is, $T'_h=0$. In this case, since $T_a\not \in R^{\T}(X)$,
the only option is that $X$ is on the right edge of $\W_{T_h}$. But then
$R^{\D}(X)\cap \add T$ is empty: If there was a $\D$-map $T_d\to X$ for some
$T$-summand $T_d$, then $T_a$ and $T_d$ would have an extension, as can be
seen from an argument similar to the above.
\end{proof}

By virtue of the preceding considerations, we can now associate to each
indecomposable object $X\in \F \backslash \add \tau T$ a unique indecomposable
$\Lambda_T$-module $M(\sigma_X)$ where
\[
\sigma_X = \left\{
\begin{array}{ll}
\sigma^{\T}_X & \textrm{if } \sigma^{\D}_X \textrm{ is zero} \\[4pt]
\sigma^{\D}_X & \textrm{if } \sigma^{\T}_X \textrm{ is zero} \\[4pt]
(\sigma^{\D}_X)^{-1} \beta_X \sigma^{\T}_X & \textrm{if both }\sigma^{\T}_X
\textrm{ and }\sigma^{\D}_X \textrm{ are non-zero}
\end{array}
\right.
\]

We can now describe the action of the Hom-functor on objects in $\F$.

\begin{theorem} \label{thm:actiononF}
Let $T$ be a maximal rigid object of $\C$, and $\Lambda_T = \End_{\C}(T)\op$
the endomorphism ring.
\begin{itemize}
\item[(1)] Let $X$ be an object in $\F \backslash \add \tau
  T$. Then the $\Lambda_T$-module $\Hom_{\C}(T,X)$ is isomorphic to the
  string module $M(\sigma_X)$.
\item[(2)] The functor $\Hom_{\C}(T,-)$ provides a bijection between
  $\F \backslash \add \tau T$ and the set of isoclasses of indecomposable
  $\Lambda_T$-modules.
\end{itemize}
\end{theorem}

\begin{proof}
\textbf{(1)} By Lemma \ref{lem:revhamm} part (i), the module is non-zero.
Let $e_i$ be the idempotent corresponding to the indecomposable
projective $P_i=\Hom_{\C_{\T}}(T,T_i)$. Let $e_i$ be the idempotent of
$\Lambda_T$ corresponding to the vertex $i$, which again corresponds to the
summand $T_i$ of $T$. Then the vector space $\Hom_{\C_{\T}}(T,X)$ decomposes
\[
\Hom_{\C_{\T}}(T,X) = \bigoplus_{i=1}^{n-1} e_i\Hom_{\C_{\T}}(T,X) =
\bigoplus_{i=1}^{n-1} \Hom_{\C_{\T}}(T_i,X) = \bigoplus_{i=1}^{n-1} \left(
  \Phi_{iX}\oplus \Psi_{iY} \right)
\]
where each vector space $\Phi_{iX}$ and $\Psi_{iX}$ is at most 1-dimensional
and is spanned by a $\T$-map $\phi_{iX}:T_i \to X$ and a $\D$-map
$\psi_{iX}:T_i \to X$ respectively, where these maps are zero if no non-zero
such maps exist.

By the definition of $\sigma^{\T}_X$, the vertices for which $\Phi_{iX}\not =
0$ are exactly the vertices that are traversed by $\sigma^{\T}_X$. Similarly,
the vertices for which $\Psi_{iX}\not = 0$ are the ones traversed by
$\sigma^{\D}_X$. In particular, there is an equality of dimension vectors
\[
\underline{\dim} \left( M(\sigma_X)\right) = \underline{\dim} \left(
  \Hom_{\C_{\T}}(T,X) \right).
\]

We need to establish that the action of $\Lambda_T$ on $\Hom_{\C_{\T}}(T,X)$
is the same as the action on $M(\sigma_X)$.

Each map which is irreducible in $\add_{\C} T$ corresponds to an arrow in the
quiver of $\Lambda_T$, and the arrow acts by composition with the irreducible
map. Unless both the start vertex and the end vertex of this
arrow are vertices in the support of $\Hom_{\C_{\T}}(T,X)$, then clearly this
map (equivalently, this arrow) has a zero action on both the modules
$\Hom_{\C}(T,X)$ and $M(\sigma_X)$.

So we must show that for each $\T$-arrow $i\to j$ appearing in
$\sigma^{\T}_X$ acts by an isomorphism $\Phi_{iX}\to \Phi_{jX}$, and each
$\T$-arrow appearing in $\sigma^{\D}_X$ acts by an isomorphism $\Psi_{iX}\to
\Psi_{jX}$, and finally that if $\beta_X:i\to j$ is defined, then the action
of this is given by an isomorphism $\Phi_{iX}\to \Psi_{jX}$.

Our first goal is now to show that whenever $i \overset{\alpha}{\to} j$ is a
$\T$-arrow such that $\alpha$ itself or $\alpha^{-1}$ appears in $\sigma_X$,
then the action of $\alpha$ is given by a pair of linear transformations
\begin{eqnarray*}
\alpha' & : & \Phi_{iX} \to \Phi_{jX} \\
\alpha'' & : & \Psi_{iX} \to \Psi_{jX}
\end{eqnarray*}
which are isomorphisms when their domains and codomains are both non-zero. (And
necessarily zero otherwise.) Let $\phi_{ji}:T_j\to T_i$ be the irreducible
$\T$-map corresponding to $\alpha$. Then what we need is that if
$\phi_{jX}$ and $\phi_{iX}$ are both non-zero, then $\phi_{ji}\cdot
\phi_{iX}=\phi_{iX}\circ \phi_{ji}=\phi_{jX}$ up to a non-zero scalar, and
similarly that if $\psi_{iX}$ and $\psi_{jX}$ are both non-zero, then
$\phi_{ji}\cdot \psi_{iX}=\psi_{iX}\circ \phi_{ji}= \psi_{jX}$. The first
assertion is clearly true by the structure of the tube. The second assertion
holds by an application of Lemma \ref{lem:rayfactoring} part (i), and Remark
\ref{rem:rayfactoring}, where we use that $\phi_{ji}$ must be a composition of
maps which are irreducible in $\C_{\T}$, and follows a ray or coray (along the
edge of a wing), and thus all the indecomposables that $\phi_{ji}$ factors
through are also in $R^{\D}(X)$.

Next let $X$ be such that the $\D$-arrow $\beta_X: i \to j$ is
defined, and thus appears in the string $\sigma_X$. Then we know that
$\Phi_{iX}$ and $\Psi_{jX}$ are non-zero. The action of
$\beta_X$ is given by composition with a $\D$-map $\psi_{ji}:T_j\to T_i$. We
wish to show that (up to multiplication by a non-zero scalar) this action is
given by a linear transformation
\[
\beta'_X : \Phi_{iX}\oplus \Psi_{iX} 
\overset{\left(
\begin{array}{cc}
0 & 0 \\
1 & 0
\end{array}
\right) }{\longrightarrow}
\Phi_{jX}\oplus \Psi_{jX}
\]
In other words, it sends $\phi_{iX}$ to $\psi_{jX}$ and annihilates
$\psi_{iX}$. The composition $\psi_{iX}\circ \psi_{ji}$ is clearly zero, as
all compositions of two $\D$-maps are.

Consider the image $\phi_{iX}\circ \psi_{ji}$ of $\phi_{iX}$. We need that the
$\T$-map $\phi_{iX}$ does not factor through any indecomposable to which there
is no $\D$-map from $T_j$. This holds, since by Lemma \ref{lem:onecomponent},
$H^{\D}(T_j)\cap \F$ forms a rectangle-shaped subgraph of the tube, and the
map $\phi_{iX}$ cannot factor through any indecomposable outside this
subgraph. We can then conclude from Lemma \ref{lem:rayfactoring} part (ii) and
Remark \ref{rem:rayfactoring} that $\phi_{iX}\circ \psi_{ji}=\psi_{jX}$, which
is what we wanted.

It remains to show that if there exists an arrow which connects two vertices
in the support of $\Hom_{\C}(T,X)$, but which doesn't appear in $\sigma_X$,
then the action of this arrow is zero on $\Hom_{\C}(T,X)$. By Lemmas
\ref{lem:stringsubwing}, \ref{lem:revhamm} and \ref{lem:Dstrings} the only case
to consider is when $\beta_X$ is the loop vertex, and there is a $T$-subwing
triple $(T_i;T_j,T_k)$ such that $T_j\in R^{\T}(X)$ and $T_k\in R^{\D}(X)$ or
vice versa. Since the action of the arrow $\beta:j\to k$ is given by
composition with the $\D$-map $\psi_{kj}:T_k \to T_j$, we only need to study the
case where $\sigma^{\T}_X$ traverses $j$ and $\sigma^{\D}_X$ traverses $k$. So
we need to show that $\phi_{jX}\circ \psi_{kj}:T_k\to X$ is zero.

But by examining the Hom-hammocks of $T_j$ and $T_k$, we see that if there is
a $\T$-map $T_j\to X$ and a $\D$-map $T_k\to X$, then either $X$ is in
$\W_{T_i}$, which contradicts the fact that $i$ must be traversed by 
$\sigma_X$, or $\phi_{iX}$ factors through objects on the coray
$\mathbf{C}_{(n,1)}$. In the latter case, the composition must be zero, since
there are no $\D$-maps from $T_k$ to any objects on this coray.

\textbf{(2)} Counting the number of elements of $\F$, we find that it contains
$n(n-1)$ objects with quasilength less than $n$ and $\frac{1}{2}n(n-1)$ with
quasilength $n$ or more, that is, a total of $\frac{3}{2}n(n-1)$
elements. Since $T$ has $n-1$ summands, the cardinality of $\F \backslash \add
\tau T$ is
\[
\frac{3}{2}n(n-1)-(n-1)=\frac{1}{2}(3n^2-5n+2)
\]
which, by Theorem \ref{thm:finitetype}, is also the number of indecomposables
in $\mod \Lambda_T$. By Lemmas \ref{lem:differentstrings} and
\ref{lem:TstringisnotDstring}, if $X$ and $Y$ are different objects in $\F
\backslash \add T$, then $\sigma_X\not =\sigma_Y$. It then follows from part
(1) that $\Hom_{\C}(T,X)\not \simeq \Hom_{\C}(T,Y)$. So we $\Hom_{\C}(T,-)$
provides a bijection.
\end{proof}

We now turn to the indecomposables which are not in $\F$. It is easily seen
that Lemma \ref{lem:revhamm} parts (ii)-(iv) hold also for indecomposables
which are not in $\F$. So we can define $\sigma^{\T}_X$ and $\sigma^{\D}_X$
in this case as well. The following theorem now completes the description of the
action of $\Hom_{\C}(T,-)$ on objects.

\begin{theorem}
Let $X$ be an indecomposable object in $\C$, where $X\not \in \F$. Then we
have the following.
\begin{itemize}
\item[(1)] $\Hom_{\C}(T,X)=0$ if and only if $X=(n,kn-1)$, where $k\geq 2$.
\item[(2)] If $X$ is not of the type described in $(a)$, then
\[
\Hom_{\C}(T,X) = M(\sigma^{\T}_X) \amalg M(\sigma^{\D}_X)
\]
where $M(\sigma)$ is the zero module if $\sigma$ is the zero string.
\end{itemize}
\end{theorem}

\begin{proof}
\textbf{(1):} When $X\not \in \F$ we know that $\Hom_{\C}(T,X)=0$ if and only
if $\Hom_{\C}(T_1,X)=0$. There are no $\T$-maps $T_1\to X$ if and only if $X$
is on the ray $\mathbf{R}_{(n,1)}$, that is $X=(n,t)$ for some $t\geq
1$. Moreover, there are no $\D$-maps $T_1\to X$ if and only if $X$ is on the
coray $\mathbf{C}_{(n,n-1)}$. The indecomposables that are in the intersection
of $\mathbf{R}_{(n,1)}$ and $\mathbf{C}_{(n,n-1)}$ and outside $\F$ are
exactly the ones with coordinates $(n,t)$ such that $n+t\equiv n+n-1
\operatorname{mod} n$. The claim follows.

\textbf{(2):} The proof of Theorem \ref{thm:actiononF} goes through, with the
following exception, which is exactly what is needed. The action of $\beta_X$
(which in this case is always the loop, as we see from the argument for (1)
above) is zero: The $\T$-map $\phi_{1X}:T_1 \to X$ factors through (at least)
one object on the coray $\mathbf{C}_{(n,n-1)}$. We know that there are no
$\D$-maps from $T_1$ to any object on this coray. It then follows that the
composition $\phi_{1X}\circ \psi_{11}$, where $\psi_{11}$ is the
$\D$-endomorphism of $T_1$, is a zero map. The result follows.
\end{proof}

\section*{Acknowledgements}

The author is very grateful to Aslak Bakke Buan for many helpful
comments on earlier versions of this manuscript. Much of the work presented
here was carried out during a visit at the School of Mathematics, University
of Leeds. The author would like to thank Robert Marsh for his hospitality as
well as for many discussions on the subject.

\end{document}